\documentclass[reqno, 12pt]{amsart}
\usepackage{a4wide}
\usepackage{amscd}
\usepackage{enumerate}
\input xy
\xyoption{all}

\newtheorem{thm}{Theorem}[section]

\newtheorem{lem}[thm]{Lemma}
\newtheorem{re}{Remark}[section]
\newtheorem*{thmA}{Theorem A}
\newtheorem*{thmB}{Theorem B}
\newtheorem*{thmC}{Theorem C}
\newtheorem*{thmD}{Theorem D}
\newtheorem*{MT1}{Main Theorem 1}
\newtheorem*{MT2}{Main Theorem 2}

\numberwithin{equation}{section}

\def \BN{\bar{\nabla}}
\def \N{\nabla}

\def \HQ{{Q^*}^m}
\def \HQQ{{Q^*}^m=SO^0_{2,m}/SO_2 SO_m}

\newcommand{\Ad}{\mathrm{Ad}}
\newcommand{\tr}{\mathrm{tr}}

\newcommand{\R}{\mathbb{R}}

\begin{document}

\title[Reeb parallel structure Jacobi operator]{Real hypersurfaces in the complex hyperbolic quadric with Reeb parallel structure Jacobi operator}
\author{\textsc{Hyunjin Lee and Young Jin Suh*}}

\address{\newline
Hyunjin Lee
\newline Research Institute of Real and Complex Manifolds
\newline Kyungpook National University
\newline 41566 Daegu, Republic of Korea}
\email{lhjibis@hanmail.net}

\address{ \newline
Young Jin Suh
\newline Department of Mathematics
\newline College of Natural Sciences
\newline Kyungpook National University
\newline Daegu 41566, Republic of Korea}
\email{yjsuh@knu.ac.kr}

\date{}

\begin{abstract}
We introduce the notion of Reeb parallel structure Jacobi operator for real hypersurfaces in the complex hyperbolic quadric ${\HQQ}$, $m \geq 3$, and give a classification theory for real hypersurfaces in ${\HQ}$, $m \geq 3$, with Reeb parallel structure Jacobi operator.
\end{abstract}

\maketitle
\thispagestyle{empty}

\footnote[0]{ * Corresponding author. 
{The first author is supported by grant Proj. No. NRF-2019-R1I1A1A 01050300 from National Research Foundation of Korea and the second author by NRF-2018-R1D1A1B-05040381.}
}

\section{Introduction}\label{section 1}

As a dual space of $m$-dimensional complex quadric~$Q^m$, we can give a Riemannian symmetric spaces $\HQ$, which is said to be complex hyperbolic quadric. The complex hyperbolic quadric $\HQ$ is realized as the quotient manifold $G/K = SO_{2,m}^{0}/SO_{2}SO_{m}$, where the transitive group~$G:=SO^{0}_{2,m}$ of $Q^{*m}$ is given by the connected component of indefinite $(m+2){\times}(m+2)$-special orthogonal group $SO_{2,m}$ and $K:=SO_{2}SO_{m}$ is the isotropic subgroup of $G$. Then $Q^{*m}$ is the simply connected Riemannian symmetric space whose curvature tensor is the negative of the curvature tensor of $Q^m$, i.e. a complex hypersurface in the indefinite complex hyperbolic space $\mathbb C H_{1}^{m+1}$ with index~1. Accordingly, ${\HQ}$ admits two important geometric structures, a real structure~$A$ and a complex structure~$J$, which anti-commute with each other, that is, $AJ=-JA$ (see \cite{ANS}, \cite{SM}, \cite{Suh19} and \cite{SH19}).

%
%

\vskip 6pt

In addition to the complex structure $J$ there is another distinguished geometric structure on $\HQ$, namely a parallel rank two vector bundle ${\mathfrak A}$ which contains an $S^1$-bundle of real structures, that is, real structures~$A$ on the tangent spaces of $\HQ$. The set is denoted by ${\mathfrak A}_{p_{0}}=\{\lambda A_{{0}}{\vert}\, {\lambda}\in S^1{\subset}{\mathbb C}\}$, $p_{0}:=e K \in {\HQ}$. It is $\mathrm{Ad}(K)$-invariant, and generates an $G$-invariant $S^{1}$-subbundle $\mathfrak A$, the set of all complex conjugations defined on $\HQ$. Then ${\mathfrak A}$ becomes a parallel rank $2$-subbundle of $\text{End}(T{\HQ})$. This geometric structure determines a maximal ${\mathfrak A}$-invariant subbundle ${\mathcal Q}$ of the tangent bundle $TM$ of a real hypersurface $M$ in $\HQ$.  Here the notion of parallel vector bundle ${\mathfrak A}$ means that $({\bar\nabla}_XA)Y=q(X)JAY$ for any vector fields $X$ and $Y$ on $\HQ$, where  $\bar\nabla$ and ~$q$ denote a connection and a certain $1$-form defined on $T_{p}{\HQ}$, $p \in {\HQ}$ respectively (\cite{SM}).

\vskip 6pt

Recall that a nonzero tangent vector $W \in T_{p}{\HQ}$, $p \in Q^{*m}$, is called singular if it is tangent to more than one maximal flat in $\HQ$. There are two types of singular tangent vectors for the complex hyperbolic quadric~$\HQ$:
\begin{itemize}
\item If there exists a conjugation $A \in {\mathfrak A}_{p}$ such that $W \in V(A)=\{X \in T_{p}{Q^*}^m{\vert}\,AX=~X\}$, then $W$ is singular. Such a singular tangent vector is called {\it ${\mathfrak A}$-principal}.
\item If there exist a conjugation $A \in {\mathfrak A}_{p}$ and orthonormal vectors $Z_{1}$, $Z_{2} \in V(A)$ such that $W/||W|| = (Z_{1}+JZ_{2})/\sqrt{2}$, then $W$ is singular. Such a singular tangent vector is called \emph{${\mathfrak A}$-isotropic},
where $V(A)=\{X \in T_{p}{Q^*}^m{\vert}\, AX=X\}$ and $JV(A)=\{X \in T_{p}{Q^*}^m{\vert} \, AX=-X\}$ are the $(+1)$-eigenspace and $(-1)$-eigenspace for the involution $A$
on $T_{p}{Q^*}^{m}$, $p \in {Q^*}^{m}$.
\end{itemize}

\vskip 6pt

Now, let $M$ be a real hypersurface in K\"{a}hler manifold~$\widetilde M$, and denote by $(\phi, \xi, \eta, g)$ the induced almost contact metric structure of $M$. As a typical classification theorem for such real hypersurface, many geometers considered the condition that a real hypersurface~$M$ in $\widetilde M$ has {\it isometric Reeb flow}, which means that the Riemannian metric is invariant along the Reeb direction~$\xi=-JN$. Algebraically it is equivalent to the notion of commuting shape operator given by $S{\phi}={\phi}S$, where $S$ is the shape operator of $M$ defined by $\widetilde \nabla_{X}N =-SX$, $X \in TM$.

\vskip 6pt

For instance, Okumura~\cite{O} proved that the Reeb flow on a real hypersurface in complex projective space~${\mathbb C}P^m = SU_{m+1}/S(U_1U_m)$ is isometric if and only if $M$ is an open part of a tube around a totally geodesic ${\mathbb C}P^k \subset {\mathbb C}P^m$ for some $k \in \{0,\ldots,m-1\}$. For the complex 2-plane Grassmannian $G_2({\mathbb C}^{m+2})= SU_{m+2}/S(U_2U_m)$ a classification was obtained by Berndt and Suh~\cite{BS1}. The Reeb flow on a real hypersurface in $G_2({\mathbb C}^{m+2})$ is isometric if and only if $M$ is an open part of a tube around a totally geodesic $ G_2({\mathbb C}^{m+1})$ in $G_2({\mathbb C}^{m+2})$. For the complex quadric $Q^m = SO_{m+2}/SO_2SO_m$, Berndt and Suh~\cite{BS2} have obtained the following result:
\begin{thmA}\label{Theorem A}
Let $M$ be a real hypersurface of the complex quadric $Q^m$, $m\geq 3$. Then the Reeb flow on $M$ is isometric if and only if $m$ is even, say $m = 2k$, and $M$ is an open part of a tube around a totally geodesic ${\mathbb C}P^k \subset Q^{2k}$.
\end{thmA}

On the other hand, as non-compact type ambient spaces, for the complex hyperbolic space~${\mathbb C}H^m = SU_{1,m}/S(U_mU_1)$ a classification was obtained by Montiel and Romero~\cite{MR}. They proved that the Reeb flow on a real hypersurface in~${\mathbb C}H^m $ is isometric if and only if $M$ is an open part of a tube around a totally geodesic ${\mathbb C}H^k$ in ${\mathbb C}H^m$ for some $k \in \{0,\cdots,m-1\}$.
For the complex hyperbolic $2$-plane Grassmannian $G_2^{*}({\mathbb C}^{m+2})= SU_{2,m}/S(U_mU_2)$ the classification of isometric Reeb flow was obtained by Suh~\cite{S}. In this case, the Reeb flow on a real hypersurface in $G^{*}_2({\mathbb C}^{m+2})$ is isometric if and only if $M$ is an open part of a tube around a totally geodesic $G_2^{*}({\mathbb C}^{m+1}) \subset G_2^{*}({\mathbb C}^{m+2})$ or a horosphere with singular normal $JN \in {\mathfrak J}N$. The geometric construction of horospheres in a non-compact manifold of negative curvature was mainly discussed in the book due to Eberlein~\cite{E}.

\vskip 6pt

In the paper due to Suh~\cite{SC} we investigate this problem of isometric Reeb flow for the complex hyperbolic quadric ${\HQQ}$. In view of the previous results, naturally, we expected that the classification might include at least the totally geodesic ${Q^{* m-1}} \subset {\HQ}$. But, the results are quite different from our expectations. The totally geodesic submanifolds of the above type are not included. Now compared to Theorem~$\rm A$, we introduce the classification as follows:
\begin{thmB} \label{Theorem B}
Let $M$ be a real hypersurface of the complex hyperbolic quadric ${\HQ}$, $m \geq 3$. The Reeb flow on $M$ is isometric if and only if $m$ is even, say $m = 2k$, and $M$ is an open part of a tube around a totally geodesic ${\mathbb C}H^k \subset {Q^{2k}}^*$ or a horosphere whose center at infinity is $\mathfrak A$-isotropic singular.
\end{thmB}
\noindent Hereafter, we denote $(\mathcal T_{A})$ and $(\mathcal H_{A})$ such tube and horosphere given in Theorem~$\rm B$, respectively. Then we see that $(\mathcal T_{A})$ and $(\mathcal H_{A})$ should be Hopf, that is, $S\xi = \alpha \xi$, and they have the $\mathfrak A$-isotropic singular normal vector field.

\vskip 6pt

On the other hand, Jacobi fields along geodesics of a given Riemannian manifold $(\bar M,g)$ satisfy a well known differential equation. This equation naturally inspires the so-called Jacobi operator. That is, if $R$ denotes the curvature operator of $\bar M$, and $X$ is a tangent vector field to $\bar M$, then the Jacobi operator $R_X \in \text{End}(T_p {\bar M})$ with respect to $X$ at $p \in \bar M$, defined by $(R_XY)(p)=(R(Y,X)X)(p)$ for any $Y \in T_p {\bar M}$, becomes a self adjoint endomorphism of the tangent bundle $T {\bar M}$ of $\bar M$. Thus, each tangent vector field $X$ to $\bar M$ provides a Jacobi operator $R_X$ with respect to $X$. In particular, for the Reeb vector field $\xi$, the Jacobi operator $R_{\xi}$ is said to be the {\it structure Jacobi operator}.

\vskip 6pt

Actually, many geometers have considered the condition that a real hypersurface $M$ in K\"{a}hler manifolds has {\it parallel structure Jacobi operator} (or {\it Reeb parallel structure Jacobi operator}, respectively), that is,  $\nabla_{X}R_{\xi}=0$ for any tangent vector field~$X$ on $M$ (or $\nabla_{\xi}R_{\xi}=0$, respectively). In \cite{KPSS}, Ki, P\'erez, Santos and Suh have investigated the Reeb parallel structure Jacobi operator in the complex space form $M_m(c)$, $c \neq 0$, and have used it to study some principal curvatures for a tube over a totally geodesic submanifold. On the other hand, P\'erez, Jeong and Suh~\cite{PJS} have investigated Hopf real hypersurfaces $M$ in $G_2({\mathbb C}^{m+2})$ with parallel structure Jacobi operator, that is, ${\nabla}_X{R}_{\xi}=0$ for any tangent vector field $X$ on $M$. Jeong, Suh and Woo \cite{JSW}  and  P\'erez and Santos \cite{PS} have generalized such a notion to the recurrent structure Jacobi operator, that is, $({\nabla}_X{R}_{\xi})Y={\beta}(X){R}_{\xi}Y$ for a certain $1$-form $\beta$ and any vector fields $X,Y$ on $M$ in $G_2({\mathbb C}^{m+2})$ or ${\mathbb C}P^m$. In~\cite{JLS}, Jeong, Lee, and Suh have considered a Hopf real hypersurface with Codazzi type of structure Jacobi operator, $(\nabla_{X}R_{\xi})Y =  (\nabla_{Y}R_{\xi})X$, in $G_{2}(\mathbb C^{m+2})$. Moreover, P\'erez, Santos and Suh~\cite{PSS} have further investigated the property of the Lie $\xi$-parallel structure Jacobi operator in complex projective space ${\mathbb C}P^m$, that is, ${\mathcal L}_{\xi}R_{\xi}=0$. In \cite{SuPeWo} Suh, P\'{e}rez, and Woo investigated the parallelism property with respect to the structure Jacobi operator~$R_{\xi}$ defined on $M$ in the complex hyperbolic quadric $\HQQ$ and gave the following result.
\begin{thmC}
There does not exist a Hopf real hypersurface in the complex hyperbolic quadrics~${Q^{*}}^{m}$, $m \geq 3$, with parallel structure Jacobi opeator, that is, $\nabla_{X}R_{\xi}=0$ for any tangent vector field~$X$ on $M$.
\end{thmC}

\vskip 6pt


Motivated by these results, in this paper we consider the case when $R_{\xi}$ of $M$ in $\HQ$ is Reeb parallel, that is, ${\nabla}_{\xi}{R}_{\xi}=0$, and first we prove the following:
\begin{MT1}\label{Main Theorem 1}
Let $M$ be a Hopf real hypersurface in $\HQ$, $m \geq 3$, with Reeb parallel structure Jacobi operator.  Then the unit normal vector field $N$ is singular, that is, $N$ is $\mathfrak A$-isotropic or $\mathfrak A$-principal.
\end{MT1}

On the other hand, in \cite{SuhHwang} we have considered the notion of Reeb parallel shape operator~$S$ for a real hypersurface $M$ in ${Q^*}^{m}$, that is, ${\nabla}_\xi{S}=0$, and have proved:
\begin{thmD}
Let $M$ be a Hopf real hypersurface in complex hyperbolic quadric $\HQ$, $m \geq 3$, with Reeb parallel shape operator and non-vanishing Reeb curvature. Then $M$ is an open part of the following:
\begin{enumerate}[\rm (1)]
\item {a tube around the totally geodesic $\mathbb C H^{k} \subset {Q^{*}}^{2k}$, where $m=2k$,}
\item {a horosphere whose center at infinity is $\mathfrak A$-isotropic singular,}
\item {a tube around the totally geodesic Hermitian symmetric space ${Q^{*}}^{m-1}$ embedded in ${Q^{*}}^{m}$,}
\item {a horosphere in ${Q^{*}}^{m}$ whose center at infinity is the equivalence class of an $\mathfrak A$-principal geodesic in ${Q^{*}}^{m}$,}
\item {a tube around the $m$-dimensional real hyperbolic space $\mathbb R H^{m}$ which is embedded in ${Q^{*}}^{m}$ as a real space form, or otherwise}
\item {$M$ has two distinct constant principal curvature given by
\begin{equation*}
\alpha, \quad \lambda =\frac{\alpha^{2} - 2}{\alpha}
\end{equation*}
with multiplicities $m$ and $(m-1)$, respectively.}
\end{enumerate}
\end{thmD}
\noindent Using Main Theorem~1 and Theorem~$\rm D$, we give a classification for Hopf real hypersurfaces in the complex hyperbolic quadric $\HQ$ with Reeb parallel structure Jacobi operator as follows:
\begin{MT2}\label{Main Theorem 2}
Let $M$ be a Hopf real hypersurface in the complex hyperbolic quadric $\HQ$, $m \geq 3$, with Reeb parallel structure Jacobi operator. If the Reeb curvature function~$\alpha:=g(S\xi, \xi)$ is non-vanishing,
then $M$ is locally congruent to the one of the following:
\begin{enumerate}[\rm (1)]
\item a tube around the totally geodesic ${\mathbb C}H^k{\subset}{Q^{*}}^{2k}$, where $m=2k$,
\item a horosphere whose center at infinity is $\mathfrak A$-isotropic singular.
\end{enumerate}
\end{MT2}

\vskip 10pt

\section{The Complex Hyperbolic Quadric}\label{section 2}

In this section, let us introduce known results about the complex hyperbolic quadric ${Q^*}^m$. This section is due to Suh and Hwang~\cite{SH19}.

\vskip 6pt

The $m$-dimensional complex hyperbolic quadric ${\HQ}$ is the non-compact dual of the $m$-dimensional complex quadric $Q^m$, i.e.~the simply connected Riemannian symmetric
space whose curvature tensor is the negative of the curvature tensor of $Q^m$.

The complex hyperbolic quadric ${\HQ}$ cannot be realized as a homogeneous complex hypersurface of the complex hyperbolic space ${\mathbb C}H^{m+1}$. In fact, Smyth~\cite[Theorem~3(ii)]{SM2} has shown that every homogeneous complex hypersurface in ${\mathbb C}H^{m+1}$ is totally geodesic. This is in marked contrast to the situation for the complex quadric $Q^m$, which can be realized as a homogeneous complex hypersurface of the complex projective space ${\mathbb C}P^{m+1}$ in such a way that the shape operator for any unit normal vector to $Q^m$ is a real structure on the corresponding
tangent space of $Q^m$, (see \cite{R} and \cite{SH19}). Another related result by Smyth, \cite[Theorem~1]{SM2}, which states that any complex hypersurface of ${\mathbb C}H^{m+1}$ for which the square of the shape operator has constant eigenvalues (counted with multiplicity) is totally geodesic, also precludes the possibility of a model of ${Q^*}^m$ as a complex
hypersurface of ${\mathbb C}H^{m+1}$ with the analogous property for the shape operator.

\vskip 6pt

Therefore we realize the complex hyperbolic quadric ${Q^*}^m$ as the quotient manifold $SO^0_{2,m}/SO_2 SO_m$. As ${Q^*}^1$ is isomorphic to the real hyperbolic space
$\mathbb{R}H^2 = SO^0_{1,2}/SO_2$, and ${Q^*}^2$ is isomorphic to the Hermitian product of complex hyperbolic spaces $\mathbb{C}H^1 \times \mathbb{C}H^1$, we suppose $m\geq 3$
in the sequel and throughout this paper. Let $G:= SO^0_{2,m}$ be the transvection group of ${Q^*}^m$ and $K := SO_2 SO_m$ be the isotropy group of ${Q^*}^m$ at the ``origin''
$p_0 := eK \in {Q^*}^m$. Then
$$
\sigma: G \to G,\; g \mapsto sgs^{-1} \quad\text{with}\quad s :=\left(
\begin{smallmatrix}
-1 & & & & & \\
& -1 & & & & \\
& & 1 & & &  \\
& & & 1 & &  \\
& & & & \ddots & \\
& & & & & 1
\end{smallmatrix}
\right)
$$
is an involutive Lie group automorphism of $G$ with $\mathrm{Fix}(\sigma)_0 = K$, and therefore ${Q^*}^m = G/K$ is a Riemannian symmetric space. The center of the isotropy group $K$ is
isomorphic to $SO_2$, and therefore ${Q^*}^m$ is in fact a Hermitian symmetric space.

\vskip 6pt

The Lie algebra $\mathfrak{g} := \mathfrak{so}_{2,m}$ of $G$ is given by
$$ \mathfrak{g} = \bigr\{ X \in \mathfrak{gl}(m+2,\mathbb{R}) \bigr| X^t \cdot s = -s \cdot X \bigr\} $$
(see \cite[p.~59]{Kna}). In the sequel we will write members of $\mathfrak{g}$ as block matrices with respect to the decomposition $\mathbb{R}^{m+2}=\mathbb{R}^2 \oplus \mathbb{R}^m$, i.e.~in the form
$$ X = \left( \begin{smallmatrix} X_{11} & X_{12} \\ X_{21} & X_{22} \end{smallmatrix} \right) \;, $$
where $X_{11}$, $X_{12}$, $X_{21}$, $X_{22}$ are real matrices of dimensions $2\times 2$, $2\times m$, $m\times 2$ and $m\times m$, respectively. Then
$$ \mathfrak{g} = \left\{ \; \left. \left( \begin{smallmatrix} X_{11} & X_{12} \\ X_{21} & X_{22} \end{smallmatrix} \right) \;\right|\; X_{11}^t=-X_{11}, \; X_{12}^t = X_{21},\; X_{22}^t = -X_{22} \;\right\} \; . $$
The linearization \,$\sigma_L=\mathrm{Ad}(s): \mathfrak{g}\to\mathfrak{g}$\, of the involutive Lie group automorphism \,$\sigma$\, induces the Cartan decomposition $\mathfrak{g} = \mathfrak{k}
\oplus \mathfrak{m}$, where the Lie subalgebra
\begin{equation*}
\begin{split}
 \mathfrak{k} &= \mathrm{Eig}(\sigma_{L},1) = \{ X \in \mathfrak{g} | sXs^{-1}=X\} \\
  & = \left\{ \; \left. \left( \begin{smallmatrix} X_{11} & 0 \\ 0 & X_{22} \end{smallmatrix} \right) \;\right|\; X_{11}^t=-X_{11}, \; X_{22}^t = -X_{22} \;\right\}\\
  & \cong  \mathfrak{so}_2 \oplus \mathfrak{so}_m
\end{split}
\end{equation*}
is the Lie algebra of the isotropy group $K$, and the $2m$-dimensional linear subspace
$$
 \mathfrak{m} = \mathrm{Eig}(\sigma_{L},-1) = \{ X \in \mathfrak{g} | sXs^{-1}=-X\}\\
  = \left\{ \; \left. \left( \begin{smallmatrix} 0 & X_{12} \\ X_{21} & 0 \end{smallmatrix} \right) \;\right|\; X_{12}^t
  = X_{21} \;\right\}
 $$
is canonically isomorphic to the tangent space $T_{p_0}{Q^*}^m$. Under the identification $T_{p_0}{Q^*}^m \cong \mathfrak{m}$, the Riemannian metric $g$ of ${Q^*}^m$ (where the constant factor of the metric is chosen so that the formulae become as simple as possible) is given by
$$ g(X,Y) = \frac{1}{2}\,\tr(Y^t \cdot X) = \tr(Y_{12}\cdot X_{21}) \quad\text{for}\quad X,Y \in \mathfrak{m}.$$
$g$ is clearly $\Ad(K)$-invariant, and therefore corresponds to an $G$-invariant Riemannian metric on ${Q^*}^m$.
The complex structure $J$ of the Hermitian symmetric space is given by
$$ JX = \Ad(j)X \quad\text{for}\quad X \in \mathfrak{m}, \quad \text{where}\quad j := \left( \begin{smallmatrix} 0 & 1 & & & & \\ -1 & 0 & & & & \\ & & 1 & & & \\ & & & 1 & & \\ & & & & \ddots & \\ & & & & & 1
\end{smallmatrix} \right) \in K \; .
$$
Because $j$ is in the center of $K$, the orthogonal linear map $J$ is $\Ad(K)$-invariant, and thus defines an $G$-invariant Hermitian structure on ${Q^*}^m$. By identifying the
multiplication with the unit complex number $i$ with the application of the linear map $J$, the tangent spaces of ${Q^*}^m$ thus become $m$-dimensional complex linear spaces,
and we will adopt this point of view in the sequel.

\vskip 6pt

As for the complex quadric, there is another important structure on the tangent bundle of the complex quadric besides the Riemannian metric and the complex structure,
namely an $S^1$-bundle $\mathfrak{A}$ of real structures. The situation here differs from that of the complex quadric in that for ${Q^*}^m$, the real structures in $\mathfrak{A}$ cannot be
interpreted as the shape operators of a complex hypersurface in a complex space form, but as the following considerations will show, $\mathfrak{A}$ still plays an important role in the description
of the geometry of ${Q^*}^m$.

\vskip 6pt

Let
$$
a_0 := \left( \begin{smallmatrix} 1 & & & & & \\ & -1 & & & & \\ & & 1 & & & \\ & & & 1 & &  \\ & & & & \ddots & \\ & & & & & 1 \end{smallmatrix} \right) \; .
$$
Note that we have $a_0 \not\in K$, but only $a_0 \in O_2\,SO_m$. However, $\Ad(a_0)$ still leaves $\mathfrak{m}$ invariant, and therefore defines an $\mathbb{R}$-linear map $A_0$ on the tangent space
$\mathfrak{m} \cong T_{p_0}{Q^*}^m$. $A_0$ turns out to be an involutive orthogonal map with $A_0 \circ J = -J \circ A_0$\, (i.e.~$A_0$ is anti-linear with respect to the complex structure of $T_{p_0}{Q^*}^m$), and hence a real structure on $T_{p_0}{Q^*}^m$. But $A_0$ commutes with $\Ad(g)$ not for all $g \in K$, but only for $g \in SO_m \subset K$. More specifically, for $g=(g_1,g_2) \in K$ with $g_1 \in SO_2$ and $g_2 \in SO_m$, say $g_1 = \left( \begin{smallmatrix} \cos(t) & -\sin(t) \\ \sin(t) & \cos(t) \end{smallmatrix} \right)$ with $t \in \R$ (so that $\Ad(g_1)$ corresponds to multiplication with the complex number $\mu := e^{it}$), we have
$$ A_0 \circ \Ad(g) = \mu^{-2} \cdot \Ad(g) \circ A_0 \; . $$
This equation shows that the object which is \,$\Ad(K)$-invariant and therefore geometrically relevant is not the real structure $A_0$ by itself, but rather the ``circle of real structures''
$$ \mathfrak{A}_{p_0} := \{ \lambda\,A_0 | \lambda \in S^1 \} \; . $$
$\mathfrak{A}_{p_0}$ is $\Ad(K)$-invariant, and therefore generates an $G$-invariant $S^1$-subbundle $\mathfrak{A}$ of the endomorphism bundle $\mathrm{End}(T{Q^*}^m)$, consisting of real structures on the tangent spaces of ${Q^*}^m$. For any $A \in \mathfrak{A}$, the tangent line to the fibre of $\mathfrak{A}$ through $A$ is spanned by $JA$.

\vskip 6pt

For any $p\in {Q^*}^m$ and $A \in \mathfrak{A}_p$, the real structure $A$ induces a splitting
$$ T_p{Q^*}^m = V(A) \oplus JV(A) $$
into two orthogonal, maximal totally real subspaces of the tangent space $T_p{Q^*}^m$. Here $V(A)$ (resp., $JV(A)$) is the $(+1)$-eigenspace (resp., ~the $(-1)$-eigenspace) of $A$. For every unit vector $Z \in T_p{Q^*}^m$ there exist $t\in [0,\frac{\pi} {4}]$, $A \in \mathfrak{A}_p$ and orthonormal vectors $Z_{1}$, $Z_{2} \in V(A)$ so that
$$ Z = \cos(t) Z_{1} + \sin(t) JZ_{2} $$
holds; see \cite[Proposition~3]{R}. Here $t$ is uniquely determined by $Z$. The vector $Z$ is singular, i.e.~contained in more than one Cartan subalgebra of $\mathfrak{m}$, if and only if
either $t=0$ or $t=\tfrac\pi4$ holds. The vectors with $t=0$ are called \emph{$\mathfrak{A}$-principal}, whereas the vectors with $t=\tfrac\pi4$ are called \emph{$\mathfrak{A}$-isotropic}.
If $Z$ is regular, i.e.~$0<t<\tfrac\pi4$ holds, then also $A$ and $Z_{1}$, $Z_{2}$ are uniquely determined by $Z$.

\vskip 6pt

As for the complex quadric, the Riemannian curvature tensor $\bar R$ of ${Q^*}^m$ can be fully described in terms of the ``fundamental geometric structures'' $g$, $J$ and $\mathfrak{A}$. In fact, under the correspondence $T_{p_0}{Q^*}^m \cong \mathfrak{m}$, the curvature ${\bar R}(X,Y)Z$ corresponds to $-[[X,Y],Z]$ for $X,Y,Z \in \mathfrak{m}$, see \cite[Chapter~XI, Theorem~3.2(1)]{KO}.
By evaluating the latter expression explicitly, one can show that one has
\begin{equation}\label{RCTHQ}
\begin{split}
{\bar R}(X,Y)Z & = -g(Y,Z)X + g(X,Z)Y \\
& \quad \,  - \, g(JY,Z)JX + g(JX,Z)JY + 2g(JX,Y)JZ \\
& \quad \,  - \, g(AY,Z)AX + g(AX,Z)AY \\
& \quad \,  - \,g(JAY,Z)JAX + g(JAX,Z)JAY
\end{split}
\end{equation}
for arbitrary $A \in \mathfrak{A}_{p_0}$. Therefore the curvature of ${Q^*}^m$ is the negative of that of the complex quadric $Q^m$, compare \cite[Theorem~1]{R}. This confirms that the
symmetric space ${Q^*}^m$ which we have constructed here is indeed the non-compact dual of the complex quadric.

\vskip 10pt

\section{Some General Equations}\label{section 3}

Let $M$ be a  real hypersurface in the complex hyperbolic quadric ${Q^*}^m$.
For any vector field $X$ on $M$ in $\HQ$, we may decompose $JX$ as
\begin{equation*}
JX={\phi}X+{\eta}(X)N
\end{equation*}
where $N$ denotes a unit normal vector field to $M$, the vector field ${\xi}=-JN$ is said to be {\it Reeb} vector field, and the $1$-form $\eta$ is given by ${\eta}(X)=g({\xi},X)$.
Then naturally $M$ admits an almost contact metric structure $(\phi,\xi,\eta,g)$ induced from the K\"ahler structure $(J, g)$ of ${Q^*}^m$ satisfying
$${\phi}^2=-I+{\eta}{\otimes}{\xi},\quad {\phi}{\xi}=0,\quad {\eta}({\xi})=1.$$
The tangent bundle $TM$ of $M$ splits orthogonally into  $TM = {\mathcal C} \oplus {\mathbb R}\xi$, where ${\mathcal C} = {\rm ker}(\eta)$ is the maximal complex subbundle of $TM$. The structure tensor field $\phi$ restricted to ${\mathcal C}$ coincides with the complex structure $J$ restricted to ${\mathcal C}$.

\vskip 6pt

At each point $p \in M$ we again define the maximal ${\mathfrak A}$-invariant subspace of $T_{p}M$
\[
{\mathcal Q}_p = \{X \in T_{p}M \mid AX \in T_{p}M\ {\rm for\ all}\ A \in {\mathfrak A}_{p}\}.
\]
\begin{lem}\label{lemma 3.1}
For each $p \in M$ we have
\begin{itemize}
\item[(i)] If $N_p$ is ${\mathfrak A}$-principal, then ${\mathcal Q}_{p} = {\mathcal C}_{p}$.
\item[(ii)] If $N_{p}$ is not ${\mathfrak A}$-principal, there exist a conjugation $A \in {\mathfrak A}$ and orthonormal vectors $X,Y \in V(A)$ such that $N_{p} = \cos(t)X + \sin(t)JY$ for some $t \in (0,\pi/4]$.
Then we have ${\mathcal Q}_{p} = {\mathcal C}_{p} \ominus {\mathbb C}(JX + Y)$.
\end{itemize}
\end{lem}

\begin{proof}
First assume that $N_p$ is ${\mathfrak A}$-principal. Then there exists a conjugation $A \in {\mathfrak A}$ such that $N_p \in V(A)$, that is, $AN_{p} = N_{p}$. Then we have $A\xi_{p} = -AJN_{p} = JAN_{p} = JN_{p} = -\xi_{p}$. It follows that $A$ restricted to ${\mathbb C}N_{p}$ is the orthogonal reflection in the line ${\mathbb R}N_{p}$. Since all conjugations in ${\mathfrak A}$ differ just by a rotation on such planes we see that ${\mathbb C}N_{p}$ is invariant under ${\mathfrak A}$. This implies that ${\mathcal C}_{p} = T_{p}{Q^ * }^{m} \ominus {\mathbb C}N_{p}$ is invariant under ${\mathfrak A}$, and hence ${\mathcal Q}_{p} = {\mathcal C}_{p}$.

Now assume that $N_p$ is not ${\mathfrak A}$-principal. Then there exist a conjugation $A \in {\mathfrak A}$ and orthonormal vectors $X,Y \in V(A)$ such that $N_{p} = \cos(t)X + \sin(t)JY$ for some $t \in (0,\pi/4]$. The conjugation $A$ restricted to ${\mathbb C}X \oplus {\mathbb C}Y$ is just the orthogonal reflection in ${\mathbb R}X \oplus {\mathbb R}Y$. Again, since all conjugations in ${\mathfrak A}$ differ just by a rotation on such invariant spaces we see that ${\mathbb C}X \oplus {\mathbb C}Y$ is invariant under ${\mathfrak A}$. This implies that ${\mathcal Q}_{p} = T_{p}Q^{*m} \ominus ({\mathbb C}X \oplus {\mathbb C}Y) = {\mathcal C}_{p} \ominus {\mathbb C}(JX+Y)$ is invariant under ${\mathfrak A}$, and hence ${\mathcal Q}_{p} = {\mathcal C}_{p} \ominus {\mathbb C}(JX+Y)$.
\end{proof}

\vskip 6pt

We see from the previous lemma that the rank of the distribution ${\mathcal Q}$ is in general not constant on~$M$. However, if $N_p$ is not ${\mathfrak A}$-principal, then $N$ is not ${\mathfrak A}$-principal in an open neighborhood of $p \in M$, and ${\mathcal Q}$ defines a regular distribution in an open neighborhood of~$p$.

\medskip
We are interested in real hypersurfaces for which both ${\mathcal C}$ and ${\mathcal Q}$ are invariant under the shape operator~$S$ of $M$. Real hypersurfaces in a K\"{a}hler manifold for which the maximal complex subbundle is invariant under the shape operator are known as Hopf hypersurfaces. This condition is equivalent to the Reeb flow on $M$, that is, the flow of the structure vector field $\xi$, to be geodesic.
We assume now that $M$ is a Hopf hypersurface. Then the shape operator $S$ of $M$ in $\HQ$ satisfies
\[
S\xi = \alpha \xi
\]
for the Reeb vector field $\xi$ and the smooth function $\alpha = g(S\xi,\xi)$ on $M$. Then we now consider the Codazzi equation
\begin{equation*}
\begin{split}
g((\nabla_XS)Y - (\nabla_YS)X,Z) & =  -\eta(X)g(\phi Y,Z) + \eta(Y) g(\phi X,Z) + 2\eta(Z) g(\phi X,Y) \\
& \quad \ \   - g(X,AN)g(AY,Z) + g(Y,AN)g(AX,Z)\\
& \quad \ \   - g(X,A\xi)g(J AY,Z) + g(Y,A\xi)g(JAX,Z).
\end{split}
\end{equation*}
Putting $Z = \xi$ we get
\begin{equation*}
\begin{split}
g((\nabla_XS)Y - (\nabla_YS)X,\xi) & =    2 g(\phi X,Y) - g(X,AN)g(Y,A\xi) + g(Y,AN)g(X,A\xi)\\
& \quad \ \  + g(X,A\xi)g(JY,A\xi) - g(Y,A\xi)g(JX,A\xi).
\end{split}
\end{equation*}
On the other hand, we have
\begin{equation}\label{eq: 3.1}
\begin{split}
& g((\nabla_XS)Y - (\nabla_YS)X,\xi) \\
& \quad =  g((\nabla_XS)\xi,Y) - g((\nabla_YS)\xi,X) \\
& \quad =  (X\alpha)\eta(Y) - (Y\alpha)\eta(X) + \alpha g((S\phi + \phi
S)X,Y) - 2g(S \phi SX,Y).
\end{split}
\end{equation}
Comparing the previous two equations and putting $X = \xi$ yields
\begin{equation}\label{e31}
Y\alpha  =  (\xi \alpha)\eta(Y)  + 2g(\xi,AN)g(Y,A\xi) -
2g(Y,AN)g(\xi,A\xi).
\end{equation}
Reinserting this into \eqref{eq: 3.1} yields
\begin{equation*}
\begin{split}
 g((\nabla_XS)Y - (\nabla_YS)X,\xi) & =   2g(\xi,AN)g(X,A\xi)\eta(Y) - 2g(X,AN)g(\xi,A\xi)\eta(Y) \\
& \quad  - 2g(\xi,AN)g(Y,A\xi)\eta(X) + 2g(Y,AN)g(\xi,A\xi)\eta(X) \\
& \quad  + \alpha g((\phi S + S\phi)X,Y) - 2g(S \phi SX,Y) .
\end{split}
\end{equation*}
Altogether this implies
\begin{equation*}
\begin{split}
0 & =  2g(S \phi SX,Y) - \alpha g((\phi S + S\phi)X,Y) + 2 g(\phi X,Y) \\
& \quad - g(X,AN)g(Y,A\xi) + g(Y,AN)g(X,A\xi)\\
& \quad + g(X,A\xi)g(JY,A\xi) - g(Y,A\xi)g(JX,A\xi)\\
& \quad - 2g(\xi,AN)g(X,A\xi)\eta(Y) + 2g(X,AN)g(\xi,A\xi)\eta(Y) \\
& \quad + 2g(\xi,AN)g(Y,A\xi)\eta(X) - 2g(Y,AN)g(\xi,A\xi)\eta(X).
\end{split}
\end{equation*}
At each point $p \in M$ we can choose $A \in {\mathfrak A}_p$ such that
\begin{equation}\label{normal vector}
N = \cos(t)Z_1 + \sin(t)JZ_2
\end{equation}
for some orthonormal vectors $Z_1$, $Z_2 \in V(A)$ and $0 \leq t \leq \frac{\pi}{4}$ (see Proposition~3 in \cite{R}). Note that $t$ is a function on $M$.
First of all, since $\xi = -JN$, we have
\begin{equation*}
\left \{
\begin{array}{l}
\xi = \sin (t)Z_2 - \cos (t)JZ_1, \\
AN  =  \cos (t)Z_1 - \sin (t)JZ_2, \\
A\xi =  \sin (t)Z_2 + \cos (t)JZ_1.
\end{array}
\right.
\end{equation*}
This implies $g(\xi,AN) = 0$ and hence
\begin{equation}\label{eq: 3.4}
\begin{split}
0 & =  2g(S \phi SX,Y) - \alpha g((\phi S + S\phi)X,Y) + 2 g(\phi X,Y) \\
& \quad  - g(X,AN)g(Y,A\xi) + g(Y,AN)g(X,A\xi)\\
& \quad  + g(X,A\xi)g(JY,A\xi) - g(Y,A\xi)g(JX,A\xi)\\
& \quad  + 2g(X,AN)g(\xi,A\xi)\eta(Y) - 2g(Y,AN)g(\xi,A\xi)\eta(X).
\end{split}
\end{equation}
We have $JA\xi = -AJ\xi = - AN$, and inserting this into \eqref{eq: 3.4} implies:
\begin{lem}\label{lemma 3.2}
Let $M$ be a Hopf real hypersurface in the complex hyperbolic quadric ${Q^*}^{m}$ with (local) unit normal vector field $N$. For each point in $p \in M$ we choose $A \in {\mathfrak A}_p$ such that
$N_p = \cos(t)Z_1 + \sin(t)JZ_2$ holds
for some orthonormal vectors $Z_1$, $Z_2 \in V(A)$ and $0 \leq t \leq \frac{\pi}{4}$. Then
\begin{equation}\label{e31-(2)}
\begin{split}
0 & =  2g(S \phi SX,Y) - \alpha g((\phi S + S\phi)X,Y) +  2g(\phi X,Y) \\
& \quad - 2g(X,AN)g(Y,A\xi) + 2g(Y,AN)g(X,A\xi) \\
& \quad - 2g(\xi,A\xi) \{g(Y,AN)\eta(X) - g(X,AN)\eta(Y) \}
\end{split}
\end{equation}
holds for all vector fields $X$ and $Y$ on $M$.
\end{lem}

We can write for any vector field $Y$ on $M$ in $\HQ$
$$
AY=BY+{\rho}(Y)N,
$$
where $BY$ denotes the tangential component of $AY$ and ${\rho}(Y)=g(AY,N)$.

\vskip 6pt

By virtue of \eqref{e31}, we assert:
\begin{lem}\label{lemma singular}
Let $M$ be a Hopf real hypersurface in complex hyperbolic quadric ${Q^{*}}^{m}$, $m \geq 3$. If the Reeb curvature function~$\alpha = g(S\xi, \xi)$ is constant, then the normal vector field~$N$ should be singular, that is, $N$ is either $\mathfrak A$-isotropic or $\mathfrak A$-principal.
\end{lem}
\begin{proof}
Assume the Reeb curvature function~$\alpha = g(S \xi, \xi)$ is constant. From this, together with \eqref{e31} and $g(A\xi, N)=0$, it follows that
\begin{equation*}
g(A\xi, \xi)g(Y, AN) =0
\end{equation*}
for any $Y \in T_{p}M$, $p \in M$. The first part $g(A\xi, \xi)=0$ implies $N$ is $\mathfrak A$-isotropic. Now let us work on the open subset $\mathcal U=\{\, p \in M\,|\, \beta(p)=g(A\xi, \xi)(p) \neq 0 \}$. Then it follows that $g(AN, Y)=0$ for all $Y \in T_{p}M$, $p \in \mathcal U$. Then, for the orthonormal basis $\{e_{1}, e_{2}, \cdots, e_{2m-1}, e_{2m}:=N\}$ of $T_{p}Q^{*m}$, the tangent vector $AN \in T_{p}Q^{*m}$ given by
\begin{equation}\label{eq: 3.2}
\begin{split}
AN & = \sum_{i=1}^{2m} g(AN, e_{i})e_{i} \\
   & = \sum_{i=1}^{2m-1} g(AN, e_{i})e_{i} + g(AN, N)N \\
   & = g(AN, N)N.
\end{split}
\end{equation}
Taking the complex conjugate~$A$ to this equation and using $A^{2}=I$ and \eqref{eq: 3.2} again, we get
$$
N = A^{2}N = g(AN, N)AN = g(AN, N)^{2} N,
$$
which means that $N$ is $\mathfrak A$-principal. In fact,  from \eqref{normal vector}, we see that $g(AN, N) = \cot 2t$, $t \in [0, \frac{\pi}{4})$ on $\mathcal U$. So, $g(AN, N)= \pm 1$ leads to $t=0$. This completes the proof of our Lemma.
\end{proof}
\begin{re}
\rm By virtue of Lemma~\ref{lemma singular}, we assert that if the Reeb function~$\alpha$ is identically vanishing on $M$ then $N$ should be singular.
\end{re}

If $N$ is $\mathfrak A$-principal, that is, $A\xi=-\xi$ and $AN = N$, we have $\rho = 0$, because ${\rho}(Y):=g(Y,AN)=g(Y,N)=0$ for any tangent vector field $Y$ on $M$ in $\HQ$. So we have $AY=BY$ for any tangent vector field $Y$ on $M$ in $\HQ$. From this, together with Lemma~\ref{lemma 3.2}, we have proved:
\begin{lem}\label{lemma 3.3}
Let $M$ be a Hopf hypersurface in the complex hyperbolic quadric ${Q^*}^m$, $m \geq 3$. Then we have
\[
(2S \phi S - \alpha(\phi S + S\phi) + 2\phi) X =
 2\rho(X)(B\xi - \beta \xi) + 2g(X,B\xi - \beta \xi)\phi B\xi ,
\]
where the function $\beta$ is given by ${\beta}:=g({\xi},A{\xi})=-g(N,AN)$.
\end{lem}

If the unit normal vector field $N$ is ${\mathfrak A}$-principal, we can choose a real structure $A \in {\mathfrak A}$ such that $AN = N$. Then we have $\rho = 0$ and $\phi B\xi = -\phi \xi = 0$, and therefore
\begin{equation}\label{e32}
2S \phi S - \alpha(\phi S + S\phi) = -2\phi.
\end{equation}
If $N$ is not ${\mathfrak A}$-principal, we can choose a real structure $A \in {\mathfrak A}$ as in Lemma~\ref{lemma 3.1} and get
\begin{equation}\label{e33}
\begin{split}
\rho(X)&(B\xi - \beta \xi) + g(X,B\xi - \beta \xi)\phi B\xi \\
 & = -g(X,\phi(B\xi - \beta\xi))(B\xi - \beta \xi) + g(X,B\xi - \beta \xi)\phi (B\xi - \beta\xi)\\
 & = ||B\xi - \beta\xi||^2\{ g(X,U)\phi U -g(X,\phi U)U \} \\
 & =  \sin^2(2t)\{ g(X,U)\phi U -g(X,\phi U)U \},
\end{split}
\end{equation}
which is equal to $0$ on ${\mathcal Q}$ and equal to $\sin^2(2t)\phi X$ on ${\mathcal C} \ominus {\mathcal Q}$. Altogether we have proved:
\begin{lem}\label{lemma 3.4}
Let $M$ be a Hopf real hypersurface in the complex hyperbolic quadric ${Q^*}^m$, $m \geq 3$. Then the tensor field
\[ 2S\phi S - \alpha (\phi S + S\phi) \]
leaves ${\mathcal Q}$ and ${\mathcal C} \ominus {\mathcal Q}$ invariant and we have
\[ 2S\phi S - \alpha (\phi S + S\phi) = -2\phi \ {\rm on}\  {\mathcal Q} \]
and
\[ 2S\phi S - \alpha (\phi S + S\phi) = -2\beta^2\phi \ {\rm on}\  {\mathcal C} \ominus {\mathcal Q}, \]
where ${\beta}:=g(A{\xi},{\xi})=-\cos 2t$ as in section~\ref{section 3}.
\end{lem}

\vskip 6pt

As the tangential part of \eqref{RCTHQ}, the curvature tensor $R$ of $M$ in complex quadric $\HQ$ is defined as follows. For any $A \in {\mathfrak A}_{p_{0}}$
\begin{equation*}
\begin{split}
R(X,Y)Z & =  -g(Y,Z)X + g(X,Z)Y - g({\phi}Y,Z){\phi}X + g({\phi}X,Z){\phi}Y + 2g({\phi}X,Y){\phi}Z \\
 & \quad  - g(AY,Z)(AX)^T + g(AX,Z)(AY)^T - g(JAY,Z)(JAX)^T\\
 & \quad + g(JAX,Z)(JAY)^T + g(SY,Z)SX-g(SX,Z)SY,
\end{split}
\end{equation*}
where $(AX)^T$ and $S$ denote the tangential component of the vector field $AX$ and the shape operator of~$M$ in $\HQ$, respectively.

\vskip 6pt

\noindent From this, putting $Y=Z={\xi}$ and using $g(A{\xi},N)=0$, the structure Jacobi operator is defined by
\begin{equation*}
\begin{split}
R_{\xi}(X)&=R(X,{\xi}){\xi}\\
&= -X+{\eta}(X){\xi}-g(A{\xi},{\xi})(AX)^T+g(AX,{\xi})A{\xi}\\
& \quad \ \  +g(X,AN)(AN)^T+g(S{\xi},{\xi})SX-g(SX,{\xi})S{\xi}.
\end{split}
\end{equation*}
Then we may put the following
$$
(AY)^T=AY-g(AY,N)N.
$$

Now let us denote by ${\nabla}$ and ${\bar{\nabla}}$ the covariant derivative of $M$ and the covariant derivative of $\HQ$, respectively.
Then by using the Gauss and Weingarten formulas we can assert the following:
\begin{lem}\label{Lemma 3.5}
Let $M$ be a real hypersurface in the complex quadric $\HQ$. Then
\begin{equation}\label{e34}
\begin{split}
{\nabla}_X(AY)^T& =q(X)JAY+A{\nabla}_XY+g(SX,Y)AN\\
& \quad -g(\{q(X)JAY+A{\nabla}_XY+g(SX,Y)AN\},N)N\\
& \quad +g(AY,SX)N+g(AY,N)SX-g(SX,AY)N,
\end{split}
\end{equation}
where $q$ denote a certain $1$-form defined on $T_{p}{\HQ}$, $p \in {\HQ}$, satisfying $({\bar\nabla}_U A)V=q(U)JAV$ for any vector fields $U$, $V \in T_{p}\HQ$.
\end{lem}
\begin{proof}
First let us use the Gauss formula to $(AY)^T=AY-g(AY,N)N$. Then it follows that
\begin{equation*}
\begin{split}
{\nabla}_X(AY)^T&={\bar{\nabla}}_X(AY)^T-{\sigma}(X,(AY)^T)\\
&={\bar{\nabla}}_X\{AY-g(AY,N)N\}-g(SX,(AY)^T)N\\
&=({\bar\nabla}_XA)Y+A{\bar\nabla}_XY - g(({\bar{\nabla}}_XA)Y+A{\BN}_XY,N)N - g(AY,{\BN}_XN)N\\
& \quad \ \ -g(AY,N){\BN}_XN-g(SX,(AY)^T)N,
\end{split}
\end{equation*}
where $\sigma$ denotes the second fundamental form and $N$ the unit normal vector field on $M$ in $\HQ$.
Then from this, if we use Weingarten formula ${\bar{\nabla}}_XN=-SX$, then we get the above formula.
\end{proof}

By putting $Y={\xi}$ and using $g(A{\xi},N)=0$, we have
\begin{equation}\label{e35}
\begin{split}
{\N}_X(A{\xi})&=q(X)JA{\xi}+A{\phi}SX+{\alpha}{\eta}(X)AN\\
& \quad \ \ -\{q(X)g(JA{\xi},N)+g(A{\phi}SX,N)+{\alpha}{\eta}(X)g(AN,N)\}N.
\end{split}
\end{equation}
Moreover, let us use also Gauss and Weingarten formula to $(AN)^T=AN-g(AN,N)N$. Then it follows that
\begin{equation}\label{e36}
\begin{split}
{\N}_X(AN)^T&={\BN}_X(AN)^T-{\sigma}(X,(AN)^T)\\
&={\BN}_X\{AN-g(AN,N)N\}-{\sigma}(X,(AN)^T)\\
&=({\BN}_XA)N+A{\BN}_XN-g(({\BN}_XA)N+A{\BN}_XN,N)\\
& \quad \ \ -g(AN,{\BN}_XN)N-g(AN,N){\BN}_XN-{\sigma}(X,(AN)^T)\\
&=q(X)JAN-ASX-g(q(X)JAN-ASX,N)N+g(AN,N)SX.
\end{split}
\end{equation}

On the other hand, we know that
\begin{equation}\label{e37}
\begin{split}
X{\beta}&=X(g(A{\xi},{\xi}))\\
&=g(({\BN}_XA){\xi}+A{\BN}_X{\xi},{\xi})+g(A{\xi},{\BN}_X{\xi})\\
&=g(q(X)JA{\xi}+A{\phi}SX+g(SX,{\xi})AN,{\xi})+g(A{\xi},{\phi}SX+g(SX,{\xi})N)\\
&=2g(A{\phi}SX,{\xi}).
\end{split}
\end{equation}

\vskip 10pt

\section{Some Key Lemmas}\label{section 4}

We will now apply some results in section~\ref{section 3} to get more information on Hopf hypersurfaces for which the normal vector field is ${\mathfrak A}$-principal everywhere.
\begin{lem}\label{lemma 4.1}
Let $M$ be a Hopf hypersurface in the complex hyperbolic quadric $\HQ$, $m \geq 3$, with ${\mathfrak A}$-principal normal vector field everywhere. Then the following statements hold:
\begin{enumerate}[\rm (i)]
\item {The Reeb curvature function $\alpha$ is constant.}
\item {If $X \in {\mathcal C}$ is a principal curvature
vector of $M$ with principal curvature $\lambda$, then ${\alpha}=\pm 2 $, ${\lambda}= \pm 1 $ for ${\alpha}=2{\lambda}$ or $\phi X$ is a principal curvature vector with
principal curvature ${\mu}=\frac{\alpha\lambda - 2}{2{\lambda}-{\alpha}}$ for ${\alpha} \neq 2 {\lambda}$.}
\end{enumerate}
\end{lem}

\begin{proof}
Let $A \in {\mathfrak A}$ such that $AN = N$. Then we also have $A\xi = -\xi$. In this situation we get
\begin{equation}\label{e41}
Y\alpha  =  (\xi \alpha)\eta(Y).
\end{equation}
Since ${\rm grad}^M \alpha = (\xi \alpha) \xi$, we can compute the Hessian ${\rm Hess}^M \alpha$ by
\[
({\rm Hess}^M \alpha)(X,Y) = g(\nabla_X{\rm grad}^M \alpha, Y)
= X(\xi \alpha)\eta(Y) + (\xi \alpha) g(\phi SX,Y).
\]
As ${\rm Hess}^M \alpha$ is a symmetric bilinear form, the previous equation implies
\[
(\xi\alpha)g((S\phi + \phi S)X,Y) = 0
\]
for all vector fields $X$, $Y$ on $M$ which are tangent to the distribution ${\mathcal C}$.

\vskip 6pt

Now let us consider an open subset ${\mathcal U}=\{p \in M  \vert \, ({\xi}{\alpha})_p \neq 0\}$. Then $(S{\phi}+{\phi}S)=0$ on $\mathcal U$. Now hereafter let us continue our discussion on this open subset $\mathcal U$.
Since $AN=N$ and $A{\xi}=-{\xi}$, Lemma~\ref{lemma 3.2} and the condition $(S{\phi}+{\phi}S)=0$  imply
\begin{equation}\label{e42}
S^2{\phi}X-{\phi}X=0.
\end{equation}
From this, replacing $X$ by ${\phi}X$, it follows that
\begin{equation}\label{e43}
S^2X=X+({\alpha}^2-1){\eta}(X){\xi}.
\end{equation}
Then differentiating (\ref{e43}) and using $X{\alpha}=({\xi}{\alpha}){\eta}(X)$ give
\begin{equation}\label{e44}
\begin{split}
& ({\nabla}_XS)SY+S({\nabla}_XS)Y \\
&\quad =2{\alpha}(X{\alpha}){\eta}(Y){\xi}+({\alpha}^2-1)\{g({\nabla}_X{\xi},Y){\xi}+{\eta}(Y){\nabla}_X{\xi}\}\\
&\quad =2{\alpha}({\xi}{\alpha}){\eta}(X){\eta}(Y){\xi}+({\alpha}^2-1)\{g({\phi}SX,Y){\xi}+{\eta}(Y){\phi}SX\}.
\end{split}
\end{equation}
From this, taking skew-symmetric part and using the anti-commuting shape operator on $\mathcal U$, we have
\begin{equation}\label{e45}
\begin{split}
({\nabla}_XS)SY-({\nabla}_YS)SX+&S(({\nabla}_XS)Y-({\nabla}_YS)X)\\
&=({\alpha}^2-1)\{{\eta}(Y){\phi}SX-{\eta}(X){\phi}SY\}.
\end{split}
\end{equation}

On the other hand, the Codazzi equation in section 3 for the $\mathfrak A$-principal unit normal vector field $N$ becomes
\begin{equation}\label{e46}
\begin{split}
({\nabla}_XS)Y-({\nabla}_YS)X=&-{\eta}(X){\phi}Y+{\eta}(Y){\phi}X+2g({\phi}X,Y){\xi}\\
&+{\eta}(X){\phi}AY-{\eta}(Y){\phi}AX,
\end{split}
\end{equation}
where we have used the tangential part of $JAY={\phi}AY+{\eta}(AY)N$ for any tangent vector field $Y$ on $M$ in $\HQ$.
From this, by applying the shape operator, we can write as follows:
\begin{equation}\label{e47}
\begin{split}
S(({\nabla}_XS)Y-({\nabla}_YS)X)=&-{\eta}(X)S{\phi}Y+{\eta}(Y)S{\phi}X + 2{\alpha}g({\phi}X,Y){\xi}\\
&+{\eta}(X)S{\phi}AY-{\eta}(Y)S{\phi}AX.
\end{split}
\end{equation}
Moreover, if we differentiate $A{\xi}=-{\xi}$ from the $\mathfrak A$-principal and use the equation of Gauss, we have
\begin{equation}\label{e48}
A{\phi}SX=-{\phi}SX \quad \text{and}\quad S{\phi}AX=-S{\phi}X,
\end{equation}
where the latter formula can be obtained by the first formula and the inner product
$$g(S{\phi}AX,Z)=-g(X,A{\phi}SZ)=g(X,{\phi}SZ)=-g(S{\phi}X,Z)$$
for any tangent vector fields $X$ and $Z$ on $M$.

Substituting (\ref{e47}) into (\ref{e45}) and using (\ref{e48}) in the obtained equation, we have
\begin{equation}\label{e49}
\begin{split}
({\nabla}_XS)SY-({\nabla}_YS)SX & =({\alpha}^2-1)\big \{{\eta}(Y){\phi}SX-{\eta}(X){\phi}SY \big \}\\
& \quad \ +{\eta}(X)S{\phi}Y-{\eta}(Y)S{\phi}X-2{\alpha}g({\phi}X,Y){\xi}\\
& \quad \ -{\eta}(X)S{\phi}AY+{\eta}(Y)S{\phi}AX\\
& = ({\alpha}^2+1)\{{\eta}(Y){\phi}SX-{\eta}(X){\phi}SY\}-2{\alpha}g({\phi}X,Y){\xi}.
\end{split}
\end{equation}
Now replacing $X$ by $Z$ in (\ref{e49}) gives
\begin{equation}\label{e410}
\begin{split}
({\nabla}_ZS)SY-({\nabla}_YS)SZ & = ({\alpha}^2+1)\big \{{\eta}(Y){\phi}SZ-{\eta}(Z){\phi}SY \big \}-2{\alpha}g({\phi}Z,Y){\xi}.
\end{split}
\end{equation}
From this, by taking the inner product with $X$, we have
\begin{equation*}
\begin{split}
g(SY,({\nabla}_ZS)X)-g(SZ,({\nabla}_YS)X)& =({\alpha}^2+1)\big\{{\eta}(Y)g({\phi}SZ,X)-{\eta}(Z)g({\phi}SY,X)\big\}\\
& \quad -2{\alpha}g({\phi}Z,Y){\eta}(X).
\end{split}
\end{equation*}
Here let us use the equation of Codazzi~(\ref{e46}) for the first and the second terms in the left side of the above equation. Then it follows that
\begin{equation}\label{e411}
\begin{split}
& g(SY, ({\nabla}_XS)Z)-g(SZ,({\nabla}_XS)Y) \\
& \ \ ={\eta}(Z)g(SY,{\phi}X)-{\eta}(X)g(SY,{\phi}Z)-2{\alpha}g({\phi}Z,X){\eta}(Y)\\
& \quad \ \ -{\eta}(Z)g(SY,{\phi}AX)+{\eta}(X)g(SY,{\phi}AZ) -{\eta}(Y)g(SZ,{\phi}X)\\
& \quad \ \ +{\eta}(X)g(SZ,{\phi}Y)+2{\alpha}g({\phi}Y,X){\eta}(Z)\\
& \quad \ \ +{\eta}(Y)g({\phi}AX,SZ)-{\eta}(X)g({\phi}AY,SZ)-2{\alpha}g({\phi}Z,Y){\eta}(X)\\
& \quad \ \ +({\alpha}^2+1)\big\{{\eta}(Y)g({\phi}SZ,X)-{\eta}(Z)g({\phi}SY,X)\big\}.
\end{split}
\end{equation}
Then by using the formulas in (\ref{e48}) from ${\mathfrak A}$-principal unit normal vector field $N$ and the anti-commuting property $S{\phi}+{\phi}S=0$ on the open subset $\mathcal U$ , the equation (\ref{e410}) can be reformed as follows:
\begin{equation}\label{e412}
\begin{split}
& g(SY,({\nabla}_XS)Z)-g(SZ,({\nabla}_XS)Y)\\
& \ \ =({\alpha}^2+3)\{{\eta}(Z)g(S{\phi}X,Y)-{\eta}(Y)g(S{\phi}X,Z)\}\\
& \quad \ \  +2{\alpha}{\eta}(Y)g({\phi}X,Z)-2{\alpha}{\eta}(Z)g({\phi}X,Y)+2{\alpha}g({\phi}Y,Z){\eta}(X).
\end{split}
\end{equation}
Then the equation (\ref{e412}) can be written as follows:
\begin{equation}\label{e413}
\begin{split}
({\nabla}_XS)SY-S({\nabla}_XS)Y&=({\alpha}^2+3)\big\{g(S{\phi}X,Y){\xi}-{\eta}(Y)S{\phi}X\big\}\\
&\quad +2{\alpha}{\eta}(Y){\phi}X-2{\alpha}g({\phi}X,Y){\xi}+2{\alpha}{\eta}(X){\phi}Y.
\end{split}
\end{equation}
Finally summing up (\ref{e44}) and (\ref{e413}) gives
\begin{equation}\label{e414}
\begin{split}
({\nabla}_XS)SY& =2g(S{\phi}X,Y){\xi}+{\alpha}({\xi}{\alpha}){\eta}(X){\eta}(Y){\xi}\\
& \quad +({\alpha}^2+1){\eta}(Y){\phi}SX+{\alpha}{\eta}(Y){\phi}X-{\alpha}g({\phi}X,Y){\xi}+{\alpha}{\eta}(X){\phi}Y.
\end{split}
\end{equation}
Then by taking the inner product of (\ref{e414}) with the Reeb vector field $\xi$ and using (\ref{e41}) and the formula
$({\nabla}_XS){\xi}=(X{\alpha}){\xi}+{\alpha}{\phi}SX-S{\phi}SX$, we have
$$
S{\phi}X=0
$$
for any tangent vector field $X$ on $M$ in $\HQ$. This gives that $SX={\alpha}{\eta}(X){\xi}$. From this, applying the shape operator $S$ and using (\ref{e43}) imply
$$
S^2X={\alpha}^2{\eta}(X){\xi}=X+({\alpha}^2-1){\eta}(X){\xi},
$$
which gives $X={\eta}(X){\xi}$. This gives a contradiction, because we have assumed $m \geq 3$. So the open subset ${\mathcal U}=\{p \in M \, \vert \, ({\xi}{\alpha})_p \neq 0\}$ of $M$ is empty. This implies
${\xi}{\alpha}=0$ on $M$ by the continuity of the the Reeb curvature function $\alpha$. Then from (\ref{e41}) it follows that $X{\alpha}=({\xi}{\alpha}){\eta}(X)=0$. So the Reeb curvature function $\alpha$ is constant on~$M$.

\vskip 6pt

The remaining part of the lemma follows
easily from the equation
 \[(2{\lambda}-{\alpha})S{\phi}X=({\alpha}{\lambda}-2){\phi}X. \]
of Lemma~\ref{lemma 3.2}.
\end{proof}


Now we want to give a new lemma which will be useful to prove our main theorem
as follows:
\begin{lem}\label{lemma 4.2}
Let $M$ be a Hopf real hypersurface in the complex hyperbolic quadric $\HQ$, $m \geq 3$, such that the normal vector field $N$ is ${\mathfrak A}$-principal everywhere. Then we have the following:
\begin{itemize}
\item[(i)] ${\bar{\nabla}}_XA=0$ for any $X \in {\mathcal C}$.
\item[(ii)] $ASX=SX$ for any $X \in \mathcal C$.
\end{itemize}
\end{lem}
\begin{proof}
In order to give a proof of this lemma, let us put ${\bar{\nabla}}_XA=q(X)JA$ for any $X \in T{Q^*}^m$. Now let us differentiate
$g(AN,JN)=0$ along any $X \in T_pM$, $p \in M$. Then it follows that
\begin{equation*}
\begin{split}
0&=g(({\bar{\nabla}}_XA)N+A{\bar{\nabla}}_XN, JN)+g(AN,({\bar{\nabla}}_XJ)N+J{\bar{\nabla}}_XN)\\
&=q(X)-g(ASX,JN)-g({\xi},SX)
\end{split}
\end{equation*}
for any $X \in T_p M$, $p \in M$.  Then the $1$-form $q$ becomes
\begin{equation}\label{e414-(2)}
q(X)=-g(ASX,{\xi})+g({\xi},SX)=g(S{\xi},X)+g({\xi},SX)=2{\alpha}{\eta}(X),
\end{equation}
where we have used that the unit normal $N$ is $\mathfrak A$-principal, that is, $A{\xi}=-{\xi}$. Then this gives (i) for any $X \in {\mathcal C}$.
\par
\vskip 6pt

On the other hand, we differentiate the formula $AJN=-JAN=-JN$ along the distribution $\mathcal C$. Then by the K\"ahler structure and the expression of
${\bar{\nabla}}_XA=q(X)JA$, we have
$$q(X)JAJN-AJSX=JSX.$$
From this, together with $\rm (i)$, it follows that $-AJSX = JASX =JSX$ , which implies $ASX=SX$ for any $X \in {\mathcal C}$.
\end{proof}

Now let us assume that $M$ is a real hypersurface in the complex hyperbolic quadric ${Q^*}^m$ with isometric Reeb flow. Then the commuting shape operator $S{\phi}={\phi}S$ implies $S{\xi}={\alpha}{\xi}$, that is, $M$ is Hopf.
We will now prove that the Reeb curvature $\alpha$ of a Hopf hypersurface is constant if the normal vector field is ${\mathfrak A}$-isotropic. Assume that the unit normal vector field $N$ is ${\mathfrak A}$-isotropic everywhere. Then we have ${\beta} = g(A{\xi},{\xi})=0$ in Lemma~\ref{lemma 3.3}. So \eqref{e31} implies
\[
Y\alpha  =  (\xi\alpha)\eta(Y)
\]
for all $Y \in TM$.
Since ${\rm grad}^M \alpha =(\xi\alpha)\xi$, we can compute the Hessian ${\rm Hess}^M \alpha$ by
\begin{eqnarray*}
({\rm Hess}^M \alpha)(X,Y) & = & g(\nabla_X{\rm grad}^M \alpha , Y) \\
& = & X(\xi\alpha)\eta(Y) + (\xi\alpha)g(\phi SX,Y).
\end{eqnarray*}
As ${\rm Hess}^M \alpha$ is a symmetric bilinear form, the previous equation implies
\[
({\xi}\alpha)g((S\phi + \phi S)X,Y) = 0
\]
for all vector fields $X,Y$ on $M$ which are tangential to ${\mathcal C}$.

\vskip 6pt

Now let us assume that $S{\phi}+{\phi}S=0$. For every principal curvature vector $X \in {\mathcal C}$ such that $SX={\lambda}X$ this implies $S{\phi}X=-{\phi}SX=-{\lambda}{\phi}X$. We assume $||X|| = 1$ and put $Y = \phi X$. Using the normal vector field $N$ is $\mathfrak A$-isotropic, that is ${\beta}=0$ in Lemma~\ref{lemma 3.3}, we know that
$$-{\lambda}^2{\phi}X+{\phi}X={\rho}(X)B{\xi}+g(X,B{\xi}){\phi}B{\xi}.$$
From this, taking the inner product with ${\phi}X$ and using
$$g(X,B{\xi})=g(X,A{\xi})=-g({\phi}X,AN)=-{\rho}({\phi}X),$$
we have
\begin{equation*}
\begin{split}
-\lambda^2 + 1 &= \rho(X)\eta(B\phi X) - \rho(\phi X)\eta(BX)\\
& =  g(X,AN)^2 + g(X,A\xi)^2 =  ||X_{{\mathcal C} \ominus {\mathcal Q}}||^2 \leq 1,
\end{split}
\end{equation*}
where $X_{{\mathcal C} \ominus {\mathcal Q}}$ denotes the orthogonal projection of $X$ onto ${\mathcal C} \ominus {\mathcal Q}$.

On the other hand, from the commutativity of $S$ and $\phi$ and the above equation for $SX={\lambda}X$ it follows that
$$-{\lambda}{\phi}X=-{\phi}SX=S{\phi}X={\phi}SX={\lambda}{\phi}X.$$
This gives that the principal curvature ${\lambda}=0$. Then the above two equation give
$||X_{{\mathcal C} \ominus {\mathcal Q}}||^2 = 1$
for all principal curvature vectors $X \in {\mathcal C}$ with $||X|| = 1$. This is only possible if ${\mathcal C} = {\mathcal C} \ominus {\mathcal Q}$, or equivalently, if ${\mathcal Q} = 0$. Since $m \geq 3$ this is not possible. Hence we must have $S{\phi}+{\phi}S \neq 0$ everywhere and therefore $d\alpha(\xi) = 0$. From this, together with (\ref{e31}), we get ${\rm grad}^M \alpha = 0$. Since $M$ is connected this implies that $\alpha$ is constant. Thus we have proved:
\begin{lem} \label{lemma 4.3}
Let $M$ be a real hypersurface in the complex hyperbolic quadric ${Q^*}^{m}$, $m \geq 3$, with isometric Reeb flow and ${\mathfrak A}$-isotropic normal vector
field $N$ everywhere. Then $\alpha$ is constant.
\end{lem}

\vskip 10pt

\section{Reeb Parallel Structure Jacobi Operator and Proof of Main Theorem~1}\label{section 5}

The curvature tensor $R(X,Y)Z$ for a Hopf real hypersurface $M$ in the complex hyperbolic quadric $\HQQ$ induced from the curvature tensor of $\HQ$ is given in section~\ref{section 3}. Now the structure Jacobi operator $R_{\xi}$ from section~\ref{section 3} can be rewritten as follows:
\begin{equation}\label{e51}
\begin{split}
R_{\xi}(Y)&=R(Y,{\xi}){\xi}\\
&=-Y+{\eta}(Y){\xi}-{\beta}(AY)^T+g(AY,{\xi})A{\xi}+g(AY,N)(AN)^T\\
& \quad \ \ +{\alpha}SY-g(SY,{\xi})S{\xi},
\end{split}
\end{equation}
where we have put ${\alpha}=g(S{\xi},{\xi})$ and ${\beta}=g(A{\xi},{\xi})$, because we assume that $M$ is Hopf.  The Reeb vector field ${\xi}=-JN$ and the anti-commuting property $AJ=-JA$ gives that the function ${\beta}$ becomes ${\beta}=-g(AN,N)$. When this function ${\beta}=g(A{\xi},{\xi})$ identically vanishes, we say that a real hypersurface~$M$ in $\HQ$ is $\mathfrak A$-isotropic as in Lemma~\ref{lemma singular}.

\vskip 6pt

\noindent Here we use the assumption of being Reeb parallel structure Jacobi operator, that is, ${\nabla}_\xi R_{\xi}=0$. Then~(\ref{e51}), together with (\ref{e34}) and (\ref{e36}), gives that
\begin{equation*}\label{e52}
\begin{split}
({\nabla}_{\xi} R_{\xi})Y  &={\nabla}_\xi(R_{\xi}Y )-R_{\xi}({\nabla}_\xi Y)\\
& = -(\xi{\beta})(AY)^T +g(q(\xi)JA{\xi}+{\alpha}AN,Y)A{\xi}\\
& \quad \, -{\beta}\Big[q(\xi)JAY+A{\nabla}_{\xi}Y+ \alpha \eta(Y) AN -g\big( q(\xi)JAY+A{\nabla}_{\xi}Y \\
& \quad \quad \quad  + \alpha \eta(Y) AN, N\big)N +\alpha g(AY,\xi)N+ \alpha g(AY,N)\xi- \alpha g(\xi,(AY)^T)N\Big]\\
& \quad \, +g(AY,{\xi})\Big[q(\xi)JA{\xi}+{\alpha}AN -\big \{q(\xi)g(JA{\xi},N)+{\alpha}g(AN,N)\big \}N\Big]\\
& \quad \, +\Big[g\big(q(\xi)JAN-\alpha A\xi + \alpha g(AN,N) \xi,Y\big)(AN)^T + g(Y,(AN)^T)q(\xi)JAN \\
& \quad \quad \quad  +g(Y,(AN)^T)\big\{- \alpha A\xi + \alpha g(AN,N)\xi - g\big(q(\xi)JAN-\alpha A\xi,N\big)N\big\}\Big]\\
&\quad \, +(\xi {\alpha})SY+{\alpha}({\nabla}_{\xi}S)Y-(\xi {\alpha}^2){\eta}(Y){\xi},
\end{split}
\end{equation*}
where we have used $g(A{\xi},N)=0$.

\noindent From this, by taking the inner product with the Reeb vector field  $\xi$, and using ${\nabla}_{\xi}R_{\xi}=0$, we have
\begin{equation}\label{e53}
\begin{split}
0& = -({\xi}{\beta})g(AY,{\xi})-{\beta}\{q({\xi})g(JAY,{\xi})+g({\nabla}_{\xi}Y,A{\xi})+{\alpha}g(AY,N)\}\\
&\quad \ \ +g(q({\xi})JA{\xi}+{\alpha}AN,Y)g(A{\xi},{\xi})\\
&\quad \ \ +g(Y, AN)\big\{q({\xi}) g(JAN,{\xi})-{\alpha}g(A{\xi},{\xi})+{\alpha}g(AN,N)\big \}\\
&\quad \ \ +{\alpha}({\xi}{\alpha}){\eta}(Y)+{\alpha}g(({\nabla}_{\xi}S)Y,{\xi})-({\xi}{\alpha}^2){\eta}(Y).
\end{split}
\end{equation}
Then first, using $g(A{\xi},N)=0$ and $({\xi}{\beta})=0$ in (\ref{e37}), we have
\begin{equation}\label{e54}
0={\beta}\big\{g({\nabla}_{\xi}Y, A{\xi})-\big(q({\xi})-2{\alpha}\big)g(Y, AN)\big\}.
\end{equation}
From this we have ${\beta}=0$ or $g({\nabla}_{\xi}Y,A{\xi})=(q({\xi})-2{\alpha})g(Y,AN)$.
The first part ${\beta}=g(A{\xi},{\xi})=0$ implies $N$ is $\mathfrak A$-isotropic. Now let us work on the open subset ${\mathcal U}=\{p \in M \vert \, \beta(p)\neq 0\}$.
Then it follows that
\begin{equation}\label{e55}
g({\nabla}_{\xi}Y,A{\xi})=(q({\xi})-2{\alpha})g(Y,AN)
\end{equation}
Then by putting $Y=(AN)^T$ in (\ref{e55}), we have
\begin{equation}\label{e56}
\begin{split}
g({\N}_{\xi}(AN)^T,A{\xi})&=(q({\xi})-2{\alpha})g((AN)^T,AN)\\
&=(q({\xi})-2{\alpha})(1-{\beta}^2)\\
&=q({\xi})-2{\alpha}-q({\xi}){\beta}^2+2{\alpha}{\beta}^2
\end{split}
\end{equation}
On the other hand, by (\ref{e36}) the left term of (\ref{e56}) becomes
\begin{equation}\label{e57}
\begin{split}
g({\N}_{\xi}(AN)^T,A{\xi})&=q({\xi})g(JAN,A{\xi})-{\alpha}g(A{\xi},A{\xi})+{\alpha}g(AN,N)g({\xi},A{\xi})\\
&=q({\xi})-{\alpha}-{\alpha}{\beta}^2.
\end{split}
\end{equation}
Then from (\ref{e56}) and (\ref{e57}) it follows that
\begin{equation}\label{e58}
{\alpha}+q({\xi}){\beta}^2-3{\alpha}{\beta}^2=0.
\end{equation}
So for any $Y$ orthogonal to $A{\xi}$ by (\ref{e35}), we have
\begin{equation}\label{e59}
g({\N}_{\xi}Y,A{\xi})=-g(Y,{\N}_{\xi}A{\xi})=(q({\xi})-{\alpha})g(Y,AN).
\end{equation}
From this, comparing with (\ref{e55}), we have
\begin{equation}\label{eq: 5.10}
\alpha g(AN, Y) =0
\end{equation}
for all $Y$ orthogonal to $A{\xi}$.

\vskip 6pt

By virtue of Lemma~\ref{lemma singular} in section~\ref{section 3}, if the Reeb curvature function~$\alpha = g(S\xi, \xi)$ is vanishing, then the unit normal vector field is singular. Thus now we only consider the case $\alpha \neq 0$ on $\mathcal U$. By \eqref{eq: 5.10}, together with $AN = AJ\xi = - JA\xi = -\phi A \xi - g(A\xi, \xi)N$, it follows that $g(\phi A \xi , Y) =0 $ for all $Y  \in  \mathfrak F$. Here we denote $\mathfrak F=\{ Y \in T_{p}M \,| \,Y \bot A\xi, \ p \in \mathcal U \}$. Substituting $Y=\phi A\xi (\in \mathfrak F)$, we get $0=1-g^{2}(A\xi, \xi) = 1 - \beta^{2}$. This implies that the unit normal $N$ is $\mathfrak A$-principal on $\mathcal U$. Together with Lemma~\ref{lemma singular} and this observation give the following lemma.
\begin{lem}\label{Lemma 5.2}
Let $M$ be a Hopf real hypersurface in the complex hyperbolic quadric $\HQ$, $m\geq 3$, with Reeb parallel structure Jacobi operator. Then the unit normal vector field $N$ is $\mathfrak A$-principal or $\mathfrak A$-isotropic.
\end{lem}

By virtue of Lemma~\ref{Lemma 5.2}, we can consider two classes of real hypersurfaces in complex hyperbolic quadric $\HQ$ with Reeb parallel structure Jacobi operator: with $\mathfrak A$-principal unit normal vector field $N$ or otherwise, with
$\mathfrak A$-isotropic unit normal vector field $N$. We will consider each case in sections $6$ and $7$ respectively.
\par
\vskip 6pt

\section{Reeb Parallel Structure Jacobi Operator with $\mathfrak A$-Isotropic Normal Vector Field}\label{section 7}

In this section we assume that the unit normal vector field $N$ of a real hypersurface $M$ in the complex hyperbolic quadric $\HQQ$ is $\mathfrak A$-isotropic.  Then the normal vector field $N$ can be written as
$$N=\frac{1}{\sqrt 2}(Z_1+JZ_2)$$
for $Z_1,Z_2 \in V(A)$, where $V(A)$ denotes the $(+1)$-eigenspace of the complex conjugation $A \in {\mathfrak A}$.
Here we note that $Z_1$ and $Z_2$ are orthonormal, i.e., we have ${\Vert}Z_1{\Vert}={\Vert}Z_2{\Vert}=1$ and $Z_1{\bot}Z_2$.
Then it follows that
$$
AN=\frac{1}{\sqrt 2}(Z_1-JZ_2),\ \ AJN=-\frac{1}{\sqrt 2}(JZ_1+Z_2),\ \ \text{and}\ \ JN=\frac{1}{\sqrt 2}(JZ_1-Z_2).
$$
Then it gives that
$$
g({\xi},A{\xi})=g(JN,AJN)=0, \ \ g({\xi},AN)=0, \ \ \text{and}\ \ g(AN,N)=0.
$$
By virtue of these formulas for $\mathfrak A$-isotropic unit normal vector field, the structure Jacobi operator is given by
\begin{equation}\label{e71}
\begin{split}
{R}_{\xi}(X)&=R(X,{\xi}){\xi}\\
&=-X+{\eta}(X){\xi}+g(AX,{\xi})A{\xi}+g(JAX,{\xi})JA{\xi}\\
& \quad \ \ +g(S{\xi},{\xi})SX-g(SX,{\xi})S{\xi}.
\end{split}
\end{equation}
On the other hand, we know that $JA{\xi}=-JAJN=AJ^2N=-AN$, and $g(JAX,{\xi})=-g(AX,J{\xi})=-g(AX,N)$. Now the structure Jacobi operator $R_{\xi}$ can be rearranged as follows:
\begin{equation}\label{e72}
{R}_{\xi}(X)=-X+{\eta}(X){\xi}+g(AX,{\xi})A{\xi}+g(X,AN)AN+{\alpha}SX-{\alpha}^2{\eta}(X){\xi}.
\end{equation}
Differentiating (\ref{e72}) we obtain
\begin{equation}\label{e73}
\begin{split}
({\nabla}_YR_{\xi})X& = {\nabla}_Y(R_{\xi}(X))-R_{\xi}({\nabla}_YX)\\
& = ({\nabla}_Y{\eta})(X){\xi}+{\eta}(X){\nabla}_Y{\xi}+g(X,{\nabla}_Y(A{\xi}))A{\xi}\\
& \quad \ \ +g(X,A{\xi}){\nabla}_Y(A{\xi})+g(X,{\nabla}_Y(AN))AN+g(X,AN){\nabla}_Y(AN)\\
& \quad \ \ +(Y{\alpha})SX+{\alpha}({\nabla}_YS)X-(Y{\alpha}^2){\eta}(X){\xi}\\
& \quad \ \ -{\alpha}^2({\nabla}_Y{\eta})(X){\xi}-{\alpha}^2{\eta}(X){\nabla}_Y{\xi}.
\end{split}
\end{equation}
On the other hand, by virtue of Lemma~\ref{Lemma 3.5}, we prove the following for a Hopf hypersurface in ${Q^*}^{m}$ with $\mathfrak A$-isotropic unit normal as follows:
\par
\vskip 6pt
\begin{lem}\label{Lemma 7.1}
 \quad Let $M$ be a Hopf real hypersurface in the complex hyperbolic quadric $\HQ$, $m \geq 3$, with $\mathfrak A$-isotropic unit normal. Then
\begin{equation*}
SAN=0, \quad \text{and} \quad SA{\xi}=0.
\end{equation*}
\end{lem}
\par
\vskip 6pt
\begin{proof}
Let us denote by ${\mathcal C}\ominus{\mathcal Q}=\text{Span}[A{\xi},AN]$. Since $N$ is isotropic, $g(AN,N)=-g(A\xi, \xi) = 0$.
By differentiating $g(AN,N)=0$, and using $({\bar{\nabla}}_XA)Y=q(X)JAY$ in the introduction and the equation of Weingarten, we know that
\begin{equation*}
\begin{split}
0&=g({\bar{\nabla}}_X(AN),N)+g(AN, {\bar\nabla}_XN)\\
&=g(q(X)JAN-ASX,N)-g(AN,SX)\\
&=-2g(ASX,N).
\end{split}
\end{equation*}
Then $SAN=0$.  Moreover, by differentiating $g(A{\xi},N)=0$, and using $g(AN,N)=0$ and $g(A{\xi},{\xi})=0$, we have the following formula
\begin{equation*}
\begin{split}
0&=g({\bar{\nabla}}_X(A{\xi}),N)+g(A{\xi},{\bar{\nabla}}_XN)\\
&=g(q(X)JA{\xi}+A({\phi}SX+g(SX,{\xi})N),N)-g(SA{\xi},X)\\
&=-2g(SA{\xi},X)
\end{split}
\end{equation*}
for any $X \in T_p M$, $p \in M$, where in the third equality we have used ${\phi}AN=JAN=-AJN=A{\xi}$. Then it follows that
$$SA{\xi}=0.$$
It completes the proof of our assertion.
\end{proof}
\noindent Here let us use the equation of Gauss and Weingarten formula as follows:
\begin{equation*}
\begin{split}
{\nabla}_Y(A{\xi})&={\bar{\nabla}}_Y(A{\xi})-{\sigma}(Y,A{\xi})\\
&=({\bar\nabla}_YA){\xi}+A{\bar\nabla}_Y{\xi}-{\sigma}(Y,A{\xi})\\
&=q(Y)JA{\xi}+A\{{\phi}SY+{\eta}(SY)N\}-g(SY,A{\xi})N,\\
\end{split}
\end{equation*}
and
\begin{equation*}
\begin{split}
{\nabla}_Y(AN)&={\bar{\nabla}}_Y(AN)-{\sigma}(Y,AN)\\
&=({\bar\nabla}_YA)N+A{\bar\nabla}_YN-{\sigma}(Y,AN)\\
&=q(Y)JAN-ASY-g(SY,AN)N.\\
\end{split}
\end{equation*}
\noindent Substituting these formulas into (\ref{e73})and using our assumption, $M$ is a Hopf real hypersurface with $\mathfrak A$-isotropic singular normal vector~$N$ in $Q^{*m}$, it yields
\begin{equation}\label{e74}
\begin{split}
({\nabla}_YR_{\xi})X & = g(X, \phi SY) \xi + \eta(X) \phi SY \\
&\quad \ \    - g(A \phi SY, \xi) BX  - g(A\xi, \phi SY)BX \\
&\quad \ \    + \big \{ g(A\phi SY, X) + \alpha \eta(Y)g(AN,X) \big \} A\xi \\
&\quad \ \    + g(A\xi,X)  \big \{ B\phi SY + \alpha \eta(Y) AN  \big \}\\
&\quad \ \    - g(AX, SY)AN - g(AN, X)BSY \\
&\quad \ \    + (Y \alpha) SX + \alpha (\nabla_{Y}S)X - 2\alpha (Y\alpha) \eta(X) \xi \\
&\quad \ \    - \alpha^{2}g(X, \phi SY) \xi - \alpha^{2} \eta(X) \phi SY.
\end{split}
\end{equation}
From this and using the assumption of Reeb parallel structure Jacobi operator, it follows that
\begin{equation}\label{e76}
({\xi}{\alpha})SX+{\alpha}({\N}_{\xi}S)X-2\alpha ({\xi}\alpha){\eta}(X){\xi}=0.
\end{equation}

\begin{lem}\label{Lemma 7.2}
Let $M$ be a real hypersurface in the complex hyperbolic quadric $\HQ$, $m \geq 3$, with Reeb parallel structure Jacobi operator and non-vanishing geodesic Reeb flow. If the unit normal vector field~$N$ of $M$ is $\mathfrak A$-isotropic, then the Reeb curvature function $\alpha$ is constant. Moreover, the shape operator~$S$ should be Reeb parallel, that is, the shape operator~$S$ satisfies the property of $\nabla_{\xi} S =0$.
\end{lem}

\begin{proof}
By putting $X={\xi}$ in the equation of Codazzi in section~\ref{section 3}, we have
\begin{equation}\label{e77}
\begin{split}
({\N}_{\xi}S)Y&=({\N}_YS){\xi}-{\phi}Y+g(Y,AN)A{\xi}-g(Y,A{\xi})AN\\
&=(Y{\alpha}){\xi}+{\alpha}{\phi}SY-S{\phi}SY-{\phi}Y+g(Y,AN)A{\xi}-g(Y,A{\xi})AN.
\end{split}
\end{equation}
From this, together with (\ref{e76}), it follows that
\begin{equation}\label{e78}
\begin{split}
({\xi}&{\alpha})SX+{\alpha}(X{\alpha}){\xi} -({\xi}{\alpha}^2){\eta}(X){\xi} \\
& + \alpha \big \{{\alpha}{\phi}SX-S{\phi}SX-{\phi}X +g(X,AN)A{\xi}-g(X,A{\xi})AN \big \}=0.
\end{split}
\end{equation}
Then by taking the inner product (\ref{e78}) with the Reeb vector field $\xi$, we have ${\alpha}X{\alpha}={\alpha}({\xi}{\alpha}){\eta}(X)$. Then (\ref{e78}) gives
\begin{equation}\label{e79}
\begin{split}
({\xi}&{\alpha})SX-{\alpha}({\xi}{\alpha}){\eta}(X){\xi}+ {\alpha} \big \{{\alpha}{\phi}SX-S{\phi}SX-{\phi}X
+g(X,AN)A{\xi}-g(X,A{\xi})AN \big \}=0.
\end{split}
\end{equation}
Since the unit normal vector field $N$ is $\mathfrak A$-isotropic, Lemma~\ref{lemma 3.2} gives
$$
S{\phi}SX - \frac{\alpha}{2}({\phi}S+S{\phi})X = -{\phi}X+g(X,AN)A{\xi}-g(X,A{\xi})AN.
$$
Substituting this one into (\ref{e79}), we have
\begin{equation}\label{e710}
2({\xi}{\alpha})SX- 2{\alpha}({\xi}{\alpha}){\eta}(X){\xi}+{{\alpha}^2}({\phi}S-S{\phi})X=0.
\end{equation}

\vskip 6pt

On the other hand, by Lemma~\ref{lemma 3.2}, when a unit normal vector field~$N$ of $M$ is $\mathfrak A$-isotropic, we get
\begin{equation*}
2 S \phi SX - \alpha (\phi S + S \phi )X + 2 \phi X - 2g(AN, X) A \xi + 2 g(X, A\xi) AN =0
\end{equation*}
for any $X \in T_{p}M$, $p \in M$. For some $X_{0} \in \mathcal Q := \{ X \in TM\,|\, X \bot \xi, A{\xi}, AN \}$ such that $SX_{0}={\lambda}X_{0}$, it becomes $(2 \lambda - \alpha) S \phi X_{0} = (\alpha \lambda -2) \phi X_{0}$. Thus we obtain:
\begin{itemize}
\item {if $\alpha = 2 \lambda$, then $\lambda = \pm 1$. Moreover, $\alpha = \pm 2$.}
\item {if $\alpha \neq 2 \lambda$, then the vector $\phi X_{0}$ is also principal with eigenvalue~$\mu$, where $\mu = \frac{\alpha {\lambda}-2}{2{\lambda}-{\alpha}}$.}
\end{itemize}
From this, let us consider of two cases as follows.
\vskip 3pt

{\bf Case I.} \quad $\alpha = 2 \lambda$

\vskip 3pt

Since $S$ is symmetric, we can choose a basis $\{e_{1}=\xi, e_{2}=A\xi, e_{3}=AN, e_{4} \cdots, e_{2m-1}\}$ for $T_{p}M$ such that $S e_{i} = \lambda_{i} e_{i}$ (in particular, $\lambda_{1}=\alpha$, $\lambda_{2}=\lambda_{3}=0$). It follows that the expression of the shape operator~$S$ becomes
\begin{equation*}
\begin{split}
S& =\mathrm{diag}(\alpha, 0, 0, \lambda_{4}, \cdots, \lambda_{2m-1})\\
& = \mathrm{diag}(\pm 2, 0, 0, \pm 1, \cdots, \pm 1),
\end{split}
\end{equation*}
where $\mathrm{diag}(a_{1}, \cdots, a_{n})$ denote a diagonal matrix whose diagonal entries starting in the upper left corner are $a_{1}, \cdots, a_{n}$. From this and \eqref{e76}, we see that $M$ becomes a Hopf real hypersurface in $Q^{*m}$, $m \geq 3$, with Reeb parallel shape operator, $\nabla_{\xi}S =0$.

\vskip 6pt

{\bf Case II.} \quad $\alpha \neq 2 \lambda$

\vskip 3pt

For some unit $X_{0} \in \mathcal Q$ such that $SX_{0}={\lambda}X_{0}$, we have $S{\phi}X_{0}={\mu}{\phi}X_{0}$, ${\mu}=\frac{\alpha {\lambda}-2}{2{\lambda}-{\alpha}}$.
Then (\ref{e710}) gives
\begin{equation}\label{e711}
2({\xi}{\alpha}){\lambda}X_{0}+{\alpha}^{2}(\lambda - \mu ){\phi}X_{0} =0.
\end{equation}
From this, by taking the inner product with $X_{0}$ we have $({\xi}{\alpha}){\lambda}=0$.
Now let us consider an open subset ${\mathcal U}=\{p \in M\,{\vert}\,({\xi}{\alpha})(p) \neq 0\}$ in $M$. Then on such an open subset $\mathcal U$ the principal curvature $\lambda$ identically vanishes. Then \eqref{e711} gives that the Reeb curvature function $\alpha$ identically vanishes on $\mathcal U$. This gives a contradiction. So such an open subset $\mathcal U$ can not exist. This means that ${\xi}{\alpha}=0$ on $M$. That is, $X{\alpha}=0$ for any $X$ on $M$ in $\HQ$. From this and using our assumption, $\alpha \neq 0$, \eqref{e76} implies that $M$ has a Hopf real hypersurface with Reeb parallel shape operator in $Q^{*m}$, $m \geq 3$.
\end{proof}

\vskip 6pt

Then by Theorem~D in the introduction we can assert the
following:
\begin{thm}\label{Theorem 7.1}
Let $M$ be a real hypersurface in the complex hyperbolic quadric $\HQ$, $m \geq 3$, with Reeb parallel structure Jacobi operator. If the unit normal vector field $N$ is $\mathfrak A$-isotropic and non-vanishing Reeb curvature, then
\begin{enumerate}[\rm (1)]
\item  a tube around the totally geodesic ${\mathbb C}H^k{\subset}{Q^{*^{2k}}}$, where $m=2k$,
\item  a horosphere whose center at infinity is $\mathfrak A$-isotropic singular.
\end{enumerate}
\end{thm}

\vskip 10pt

\section{Reeb Parallel Structure Jacobi Operator with $\mathfrak A$-Principal Normal Vector Field}\label{section 8}
Let $M$ be a real hypersurface with non-vanishing geodesic Reeb flow, $\alpha \neq 0$, in the complex hyperbolic quadric~$\HQQ$, $m \geq 3$. In addition, we assume that $M$ has Reeb parallel structure Jacobi operator and $\mathfrak A$-principal normal vector field. Then the unit normal vector field $N$ satisfies $AN=N$
for a complex conjugation $A \in {\mathfrak A}$. Then it follows that $A{\xi}=-{\xi}$ and $g(A{\xi},{\xi})=-1$.

\vskip 6pt

By the assumption of Reeb parallel structure Jacobi operator, we have
\begin{equation}\label{e81}
\begin{split}
0&=({\N}_{\xi}R_{\xi})X \\
 & =\big\{q({\xi})JAX+g(SX,{\xi})N-g(S{\xi},AX)N \big\}\\
 & \quad \ \ +({\xi}{\alpha})SX+{\alpha}({\N}_{\xi}S)X-({\xi}{\alpha}^2){\eta}(X){\xi}
\end{split}
\end{equation}
From \eqref{e414-(2)}, we know $q({\xi})=2{\alpha}$. By Lemma~\ref{lemma 4.1}, the Reeb curvature function $\alpha$ is constant on~$M$. So (\ref{e81}) reduces to the following
\begin{equation*}
{\alpha}({\N}_{\xi}S)X = -2\alpha {\phi}AX.
\end{equation*}
Since $M$ has non-vanishing geodesic Reeb flow, that is, $\alpha=g(S\xi, \xi) \neq 0$, we have
\begin{equation}\label{e82}
({\N}_{\xi}S)X = -2{\phi}AX.
\end{equation}
On the other hand, by using the equation of Codazzi in section~\ref{section 3}, we have
\begin{equation*}
\begin{split}
g\big(({\N}_XS){\xi}-({\N}_{\xi}S)X,Z\big)&= g({\phi}X,Z)-g(X,AN)g(A{\xi},Z)\\
& \quad \ \ -g(X,A{\xi})g(JA{\xi},Z)+g({\xi},A{\xi})g(JAX,Z)\\
& = g({\phi}X,Z)-g({\phi}AX,Z).
\end{split}
\end{equation*}
In addition, since $M$ is Hopf, it leads to
\begin{equation*}
\begin{split}
({\N}_{\xi}S)X &= (\nabla_{X}S)\xi - \phi X + \phi AX \\
& = \alpha \phi SX - S \phi SX  - \phi X + \phi AX
\end{split}
\end{equation*}
From this, together with \eqref{e82}, it follows that
\begin{equation}\label{e83}
\alpha \phi SX - S \phi SX  - \phi X = - 3 \phi AX.
\end{equation}
By virtue of Lemma~\ref{lemma 3.2}, for the $\mathfrak A$-principal unit normal vector field, we obtain
\begin{equation}\label{e84}
2 S{\phi}SX= \alpha (S{\phi}+{\phi}S)X - 2{\phi}X.
\end{equation}
Therefore, (\ref{e83}) can be written as
\begin{equation}\label{e85}
\alpha ({\phi}S-S{\phi})X=-6 {\phi}AX.
\end{equation}
Inserting $X=SY$ for $Y \in \mathcal C$ into \eqref{e85} and taking the structure tensor~$\phi$ leads to
\begin{equation*}
-\alpha S^{2}Y  - \alpha \phi S{\phi} SY = 6 ASY,
\end{equation*}
where $\mathcal C =\mathrm{ker}\eta$ denotes the maximal complex subbundle of $TM$, which is defined by a distribution $\mathcal C = \{ Y \in T_{p}M\,|\, \eta(Y)=0\}$ in $T_{p}M$, $p \in M$.

\noindent By using \eqref{e84} and Lemma~\ref{lemma 4.2} this equation gives us
\begin{equation}\label{e86}
\alpha^{2} \phi S{\phi} Y = -2 \alpha S^{2}Y   + \alpha^{2} SY - 2 \alpha Y -12 SY
\end{equation}
for all $Y \in \mathcal C$.

\vskip 6pt

On the other hand, in this section we have assumed that the normal vector field~$N$ of $M$ is $\mathfrak A$-principal. So it follows that $AX \in TM$ for all $X \in TM$. From this, the anti-commuting property with respect to $J$ and $A$ implies $\phi A X = -A \phi X$. Hence \eqref{e85} can be expressed as
\begin{equation}\label{e87}
\alpha ({\phi}S-S{\phi})X= 6 A \phi X.
\end{equation}
Putting $X=\phi Y$ into \eqref{e87}, it gives
$$
\alpha {\phi}S \phi Y = - \alpha SY - 6 AY
$$
for all $Y \in \mathcal C$. Inserting this into \eqref{e86} gives
\begin{equation}\label{e88}
3 \alpha AY - \alpha S^{2}Y + \alpha^{2} SY - \alpha Y - 6 SY =0.
\end{equation}
Taking the complex conjugate~$A$ to \eqref{e88} again and using the second equation in Lemma~\ref{lemma 4.2}, we get
\begin{equation}\label{e89}
3 \alpha Y - \alpha S^{2}Y + \alpha^{2} SY - \alpha AY - 6 SY =0,
\end{equation}
for all $Y \in \mathcal C$. Summing up \eqref{e88} and \eqref{e89}, gives $AY=Y$ for all $Y \in \mathcal C$. This gives a contradiction. In fact, it is well known that the trace of the real structure~$A$ is zero, that is, $\mathrm{Tr} A =0$ (see Lemma~1 in \cite{SM}). For an orthonormal basis $\{ \, e_{1}, e_{2} \cdots, e_{2m-2}, e_{2m-1}=\xi, e_{2m}=N \,\}$ for $T{Q^*}^{m}$, where $ e_{j} \in \mathcal C$ $(j = 1, 2, \cdots, 2m-2)$, the trace of $A$ is given by
\begin{equation*}
\begin{split}
\mathrm{Tr}A & = \sum_{i=1}^{2m} g(A e_{i}, e_{i}) \\
&= g(AN, N) + g(A \xi, \xi) + \sum_{i=1}^{2m-2} g(A e_{i}, e_{i}) \\
&= 2m-2.
\end{split}
\end{equation*}
It implies that $m=1$. But we now consider for the case $m \geq 3$.

\vskip 6pt

Consequently, this completes the proof that {\it there does not exists a Hopf real hypersurface $(\alpha \neq 0)$ in complex hyperbolic quadrics~${Q^{*}}^{m}$, $m \geq 3$, with Reeb parallel structure Jacobi operator and $\mathfrak A$-principal normal vector field}.

\vskip 10pt



\begin{thebibliography}{99}
\bibitem{ANS} R. Aiyama, H. Nakagawa, Y.J. Suh, {\it Semi-Kaehlerian subamnifodls in an indefinite complex space form}, Kodai Math. J. {\bf 11} (1988), 325-343.
\bibitem{BS1} J.~Berndt \and Y.J.~Suh, {\it Real hypersurfaces with isometric Reeb flow in complex two-plane Grassmannians}, Monatsh. Math.~{\bf 137} (2002), 87--98.
%
\bibitem{BS2} J.~Berndt \and Y.J.~Suh, {\it Real hypersurfaces with isometric Reeb flow in complex quadrics}, Internat. J. Math.~{\bf 24} (2013), 1350050, 18 pp.
%
%
\bibitem{E} P.B. Eberlein, {\it Geometry of nonpositively curved manifolds}, University of Chicago Press,  Chicago, IL, 1996.
%
%
\bibitem{JLS} I. Jeong, H. Lee \and Y.J. Suh, {\it Real hypersurfaces in complex two-plane Grassmannians whose structure Jacobi operator is of Codazzi type}, Acta Math. Hungar.~{\bf 125} (2009), 141--160.
%
\bibitem{JSW}  I. Jeong, Y.J. Suh \and C. Woo, {\it Real hypersurfaces in complex two-plane Grassmannians with recurrent structure Jacobi operator}, Real and Complex Submanifolds, Springer Proc. Math. Stat.~{\bf 106}, Springer, Tokyo, 2014, 267--278.
%
\bibitem{KPSS} U-H. Ki, J.D. P\'erez, F.G. Santos \and Y.J. Suh, {\it Real hypersurfaces in complex space forms with $\xi$-parallel Ricci tensor and structure Jacobi operator}, J. Korean Math. Soc.~{\bf 44} (2007), 307--326.
%
\bibitem{Kna} A. W. Knapp, {\it Lie Groups Beyond an Introduction (2nd Ed.)}, Progress in Mathematics, Birkh\"auser Boston, 2002.
%
\bibitem{KO} S. Kobayashi and K. Nomizu, {\it Foundations of Differential Geometry}, Vol. II, A Wiley-Interscience Publ., Wiley Classics Library Ed., 1996.
%
\bibitem{LS} H. Lee and Y.J. Suh, {\it Real hypersurfaces in the complex quadric with Reeb parallel structure Jacobi opeator}, submitted.

\bibitem{MR} S. Montiel \and A. Romero, {\it On some real hypersurfaces of a complex hyperbolic space}, Geom. Dedicata~{\bf 20} (1986), 245--261.
%
\bibitem{O} M.~Okumura, {\it On some real hypersurfaces of a complex projective space}, Trans. Amer. Math. Soc.~{\bf 212} (1975), 355--364.
%
\bibitem{PS} J.D. P\'erez \and F.G. Santos, {\it Real hypersurfaces in complex projective space with recurrent structure Jacobi operator}, Differential Geom. Appl.~{\bf 26} (2008), 218--223.
%
\bibitem{PSS} J.D. P\'erez, F.G. Santos \and Y.J. Suh, {\it Real hypersurfaces in complex projective space whose structure Jacobi operator is Lie $\xi$-parallel}, Differential Geom. Appl.~{\bf 22} (2005), 181--188.
%
\bibitem{PJS} J.D. P\'erez, I. Jeong \and Y.J. Suh, {\it Real hypersurfaces in complex two-plane Grassmannian with parallel structure Jacobi operator}, Acta. Math. Hungar.~{\bf 122} (2009), 173--186.
%
\bibitem{R}  H.~Reckziegel, {\it On the geometry of the complex quadric}, Geometry and Topology of Submanifolds, Lect. Notes in Math.~VIII, World Sci. Publ., River Edge, NJ, 1995, 302--315.
%
\bibitem{SM} B. Smyth, {\it Differential geometry of complex hypersurfaces}, Ann. of Math.~{\bf 85} (1967), 246--266.
%
\bibitem{SM2} B. Smyth, {\it Homogeneous complex hypersurfaces}, J. Math. Soc. Japan~{\bf 20} (1968), 643--647.
%
\bibitem{S} Y.J. Suh, {\it Hypersurfaces with isometric Reeb flow in complex hyperbolic two-plane Grassmannians}, Adv. in Appl. Math.~{\bf 50} (2013), 645–-659.
%
%
%
%
%
\bibitem{SC} Y.J. Suh, {\it Real hypersurfaces in the complex hyperbolic quadric with isometric Reeb flow}, Commun. Contemp. Math.~{\bf 20} (2018), 1750031 (20 pages).

\bibitem{Suh19} Y.J. Suh, {\it Pseudo-anti commuting Ricci tensor for real hypersurfaces in the complex hyperbolic quadric}, Sci. China Math.~{\bf 62} (2019), no. 4, 679-698.
%
\bibitem{SuhHwang} Y.J. Suh \and D.H. Hwang, {\it Real hypersurfaces in the complex hyperbolic quadric with Reeb parallel shape operator}, Ann. Mat. Pura Appl.~{\bf 196} (2017), 1307-1326.
%
\bibitem{SH19} Y.J. Suh \and D.H. Hwang, {\it Real hypersurfaces in the complex hyperbolic quadric with commuting Ricci tensor}, Bull. Malays. Math. Sci. Soc.~{\bf 42} (2019), no. 3, 1173-1198.
%
\bibitem{SuPeWo} Y.J. Suh, J.D. P\'{e}rez, \and C. Woo, {\it Real hypersurfaces in the complex hyperbolic quadric with parallel sturcture Jacobi operator}, Publ. Math. Debrecen~{\bf 94} (2019), no. 1-2, 75-107.

\end{thebibliography}
\end{document}